\documentclass[a4paper]{amsart}
\usepackage{amssymb}
\usepackage{amscd}
\usepackage{amsthm}
\usepackage{amsmath}
\usepackage{mathtools}
\usepackage{latexsym}
\usepackage[all]{xy}
\usepackage[utf8]{inputenc}
\usepackage{enumitem}
\usepackage{hyperref}
\usepackage{color}

\setlist[enumerate]{label={(\roman*)}}

\theoremstyle{plain}
\newtheorem{theorem}{Theorem}
\newtheorem{corollary}[theorem]{Corollary}
\newtheorem{lemma}[theorem]{Lemma}
\newtheorem{proposition}[theorem]{Proposition}

\theoremstyle{definition}
\newtheorem{definition}[theorem]{Definition}

\theoremstyle{remark}

\newtheorem*{remark}{Remark}

\numberwithin{theorem}{section} 

\newcommand{\Spec}[1]{\operatorname{Spec}(#1)}

\newcommand{\Add}{\mbox{\rm{Add\,}}}
\newcommand{\VAss}{\mbox{\rm{VAss\,}}}
\newcommand{\Loc}{\mbox{\rm{Loc\,}}}
\newcommand{\Coloc}{\mbox{\rm{Coloc\,}}}
\newcommand{\Prod}{\mbox{\rm{Prod\,}}}

\newcommand{\Ass}[2]{\mbox{\rm{Ass}}_{#1}(#2)}

\newcommand{\Hom}[3]{\operatorname{Hom}_{#1}(#2,#3)}
\newcommand{\Ext}[4]{\operatorname{Ext}^{#1}_{#2}(#3,#4)}
\newcommand{\Tor}[4]{\mbox{\rm{Tor}}_{#1}^{#2}(#3,#4)}

\newcommand{\rmod}[1]{\mbox{\rm{Mod}-}{#1}}

\newcommand{\ProjR}{\mathrm{Proj}\text{-}R}
\newcommand{\ModR}{\mathrm{Mod}\text{-}R}
\newcommand{\ModS}{\mathrm{Mod}\text{-}S}
\newcommand{\Qcoh}[1]{\mathrm{Qcoh}({#1})}

\newcommand{\pd}[1]{\mbox{\rm{proj.dim\,}}#1}

\newcommand{\supph}[1]{\mbox{\rm{supph}} \,#1}
\newcommand{\Ker}[1]{\mbox{\rm{Ker}}(#1)}

\newcommand{\p}{\mathfrak{p}}
\newcommand{\q}{\mathfrak{q}}

\newif\ifcomments

\definecolor{MichalH}{rgb}{0.9,0.1,0.1}
\definecolor{JanS}{rgb}{0.2,0.8,0.0}
\definecolor{JanT}{rgb}{0.3,0.3,0.9}

\ifcomments
\newcommand{\MichalH}[1]{{\color{MichalH}{#1}}}
\newcommand{\JanS}[1]{{\color{JanS}{#1}}}
\newcommand{\JanT}[1]{{\color{JanT}{#1}}}
\else
\newcommand{\MichalH}[1]{}
\newcommand{\JanS}[1]{}
\newcommand{\JanT}[1]{}
\fi
\begin{document}

\title{Zariski locality of quasi-coherent sheaves associated with tilting}
\author{Michal Hrbek}
\address{Institute of Mathematics CAS, \v{Z}itn\'{a} 25, 115 67 Prague, Czech Republic}
\email{hrbek@math.cas.cz}
\author{Jan \v S\v tov\'\i\v cek}
\address{Charles University, Faculty of Mathematics and Physics, Department of Algebra \\
Sokolovsk\'{a} 83, 186 75 Prague 8, Czech Republic}
\email{stovicek@karlin.mff.cuni.cz}
\author{Jan Trlifaj}
\address{Charles University, Faculty of Mathematics and Physics, Department of Algebra \\
Sokolovsk\'{a} 83, 186 75 Prague 8, Czech Republic}
\email{trlifaj@karlin.mff.cuni.cz}
\date{\today}
\thanks{Research supported by the Eduard \v Cech Institute for Algebra, Geometry and Mathematical Physics (GA\v CR project P201/12/G028). Research of M.H. also supported by RVO: 6798584.}
\begin{abstract} A classic result by Raynaud and Gruson says that the notion of an (infinite dimensional) vector bundle is Zariski local. This result may be viewed as a particular instance (for $n = 0$) of the locality of more general notions of quasi-coherent sheaves related to (infinite dimensional) $n$-tilting modules and classes. Here, we prove the latter locality for all $n$ and all schemes. We also prove that the notion of a tilting module descends along arbitrary faithfully flat ring morphisms in several particular cases (including the case when the base ring is noetherian).
\end{abstract}


\maketitle
\tableofcontents

\section*{Introduction} 
The category $\Qcoh{X}$ of quasi-coherent sheaves on an arbitrary scheme $X$ is known to be equivalent to the category of quasi-coherent representations of the graph $G$ whose vertices are open affine subschemes $U \subseteq X$, and arrows $U \to U^\prime$ correspond to the inclusions $U^\prime \subseteq U$, see \cite{EE}. If $\mathcal O _{\!X}$ denotes the structure sheaf, then the representation corresponding to a quasi-coherent sheaf $\mathcal Q$ on $X$ assigns to each vertex $U$ the $\mathcal O _{\!X}(U)$-module of sections $\mathcal Q(U)$, and to each arrow $U \to U^\prime$ the restriction of sections $\mathcal O _{\!X}(U)$-homomorphism $\mathcal Q(U) \to \mathcal Q(U^\prime)$.

As pointed out in \cite{EE}, we do not have to consider all open affine subschemes $U$ of $X$ to form the graph above. A subset $G^\prime \subseteq G$ will do, provided that $G^\prime$ covers both $X$ and all the intersections $U \cap U^\prime$ for $U, U^\prime \in G^\prime$. Then again $\Qcoh{X}$ will be equivalent to the category of quasi-coherent representations of the graph $G^\prime$. One can view the choice of $G^\prime$ as a choice of coordinates on $X$. 

We can easily extend various notions defined for modules over commutative rings to the (global) setting of quasi-coherent sheaves on schemes using coordinates as follows. Given such notion, i.e., for each commutative ring $R$, a property $\mathfrak P _R$ of $R$-modules, and a quasi-coherent sheaf $\mathcal Q$ on a scheme $X$, we simply require that for each open affine subscheme $U \in G^\prime$, the $\mathcal O _{\!X}(U)$-module of sections $\mathcal Q(U)$ satisfy the property $\mathfrak P _{\mathcal O _{\!X}(U)}$. For example, if $\mathfrak P _R$ denotes the property of being a projective $R$-module, then the corresponding global notion is that of an (infinite dimensional) vector bundle, \cite{D}. 

Of course, when extending a module theoretic notion to quasi-coherent sheaves on schemes as above, we want the resulting global notion to be independent of the choice of coordinates. In other words, we wish to test for the global notion using some coordinates $G^\prime$, but at the same time, require it to be independent of our particular choice of these coordinates. This is exactly the requirement of the \emph{Zariski locality} of the global notion, also known as \emph{affine-locality}, \cite[5.3.2]{V}. 

A classic result by Raynaud and Gruson says that the notion of a vector bundle is Zariski local for any scheme $X$, \cite{RG}. In \cite{EGT}, this has been generalized to show that the notion of a $\kappa$-restricted Drinfeld vector bundle is Zariski local for each infinite cardinal $\kappa$ (the case of $\kappa = \aleph_0$ being the classic one, cf.\ \cite{EGPT}). 

In the present paper, we pursue a different path and consider generalizations of vector bundles suggested by (infinite dimensional) tilting theory. Our starting point is the easy fact that projective generators coincide with $0$-tilting modules, and projective modules are exactly the elements of the left class, and of the kernel, of the $0$-tilting cotorsion pair $(\ProjR, \rmod R)$. Replacing $0$-tilting modules and classes with general $n$-tilting ones (where $n \geq 0$), we obtain thus the global notion of a locally $n$-tilting quasi-coherent sheaf, and of locally left, and kernel $n$-tilting, quasi-coherent sheaves. Since tilting modules are projective objects in the tilting t-structures which they induce (see e.g.~\cite[\S1]{PS}), locally kernel $n$-tilting objects can be viewed as vector bundles in certain categories tilted from the category of quasi-coherent sheaves.

Our main result, Theorem~\ref{main} below, says that the notions of locally $n$-tilting, locally left $n$-tilting, and locally kernel $n$-tilting quasi-coherent sheaf all are Zariski local for any scheme $X$. Our proof relies on relative Mittag-Leffler conditions and their relation to tilting discovered in \cite{AH}, and on the structure theory of tilting classes involving characteristic sequences of Thomason sets in $\Spec R$, developed for commutative noetherian rings in \cite{APST}, and more recently, in \cite{HS} for general commutative rings.  

The Zariski locality of $\kappa$-restricted Drinfeld vector bundles was proved in \cite{EGPT} by showing a stronger result, namely that the corresponding property of modules, the $\kappa$-restricted Mittag-Leffler property, is an ad-property. The latter means that the property ascends along any flat ring homomorphism, and descends along any faithfully flat ring monomorphism. 

In Theorem~\ref{main-1dim}, we show that the notion of a $1$-tilting module is also an ad-property. We do not know whether this extends to $n$-tilting modules for $n > 1$ in general, the missing piece being a proof of the faithfully flat descent for $n$-tilting modules with $n>1$. However, we do prove the general descent in two particular cases---when the base ring is noetherian (Corollary~\ref{main-noetherian}), and when the corresponding characteristic sequence is ``basic'' (Theorem~\ref{main-basic}).

\section{Preliminaries}
\label{sec:prelim}

For a ring $R$, we will denote by $\rmod R$ the class of all (right $R$-) modules. A~module $M$ is called \emph{strongly finitely (countably) presented} provided that $M$ has a~projective resolution consisting of finitely (countably) generated modules. More in general, if $\kappa$ is an infinite cardinal, then $M$ is \emph{strongly $< \kappa$-presented} provided that $M$ has a projective resolution consisting of $< \kappa$-generated modules.  

A \emph{filtration} of a module $M$ is a chain $\mathcal M = (M_\alpha \mid \alpha \leq \sigma)$ of submodules of $M$ such that $M_ 0 = 0$, $M_\sigma = M$, $M_\alpha \subseteq M_{\alpha + 1}$ for each $\alpha < \sigma$, and $M_{\alpha} = \bigcup_{\beta < \alpha} M_{\beta}$ for each limit ordinal $\alpha \leq \sigma$.     

Let $\mathcal C$ be a class of modules. A \emph{$\mathcal C$-filtration} of a module $M$ is a filtration $\mathcal M$ as above, such that for each $\alpha < \sigma$, $M_{\alpha + 1}/M_{\alpha}$ is isomorphic to an element of $\mathcal C$. We will denote by $\varinjlim \mathcal C$ the class of all modules that are direct limits of direct systems consisting of modules from $\mathcal C$.

Let $\kappa$ be an infinite regular cardinal such that each module in $\mathcal C$ is strongly $<\kappa$-presented. Then each $\mathcal C$-filtration $\mathcal M$ of a module $M$ can be enlarged to a family $\mathcal F$ of submodules of $M$ such that $(\mathcal F, \subseteq)$ forms a complete distributive sublattice of the (modular) lattice of all submodules of $M$, and if $N \subseteq P$ are two modules from $\mathcal F$, then $P/N$ is $\mathcal C$-filtered; moreover, for each $N \in \mathcal F$ and each subset $X \subset M$ of cardinality $< \kappa$, there exists $P \in \mathcal F$ such that $N \cup X \subseteq P$ and $P/N$ is $< \kappa$-presented. The family $\mathcal F$ is called the \emph{Hill family} extending the $\mathcal C$-filtration $\mathcal M$, see e.g.\ \cite[7.10]{GT}. 

Let $\mathcal B$ be a class of modules. A module $M$ is \emph{$\mathcal B$-stationary}, provided that $M$ can be expressed as the direct limit of a direct system $\mathcal D$ of finitely presented modules so that for each $B \in \mathcal B$, the inverse system obtained from $\mathcal D$ by applying the contravariant functor $\Hom R{-}{B}$, is Mittag-Leffler (see e.g.\ \cite[\S3]{AH}). That is, for any $M\in\mathcal D$ there exists a homomorphism $f\colon M \to M'$ in $\mathcal D$ such that for any homomorphism $g\colon M' \to B$ to $B$ and any homomorphism $f'\colon M'\to M''$ in $\mathcal D$, the composition $gf\colon M \to B$ factors through the composition $f'f\colon M \to M''$.

We will also need notation for various orthogonal classes of the Ext and Tor bifunctors: 

$\mathcal C ^{\perp} := \mbox{Ker} \Ext 1R{\mathcal C}{-} = \{ M \in \rmod R \mid \Ext 1RCM = 0 \hbox{ for all } C \in \mathcal C \}$, ${}^{\perp} \mathcal C := \mbox{Ker} \Ext 1R{-}{\mathcal C}$, $\mathcal C ^{\perp_\infty} := \bigcap_{0 < i < \omega} \mbox{Ker} \Ext iR{\mathcal C}{-}$, $\mathcal C ^{\intercal} := \mbox{Ker} \Tor 1R{\mathcal C}{-}$, and $^{\intercal} \mathcal D  := \mbox{Ker} \Tor 1R{-}{\mathcal D}$, for a class of left $R$-modules $\mathcal D$. A pair of classes $\mathfrak C = (\mathcal A, \mathcal B)$ such that $\mathcal A = {}^\perp \mathcal B$ and $\mathcal B = \mathcal A ^\perp$ is called a \emph{cotorsion pair}; the class $\mathcal A \cap \mathcal B$ is the \emph{kernel} of $\mathfrak C$. We will use the key fact that given a set $\mathcal S$ of $R$-modules, then the double-orthogonal class ${}^\perp(\mathcal S ^\perp)$ consists precisely of all direct summands of $(\mathcal{S}\cup\{R\})$-filtered $R$-modules, \cite[6.14]{GT}.

\begin{definition}\label{tilt} Let $R$ be a ring and $n < \omega$. A (right $R$-) module $T$ is \emph{$n$-tilting} provided that
\begin{itemize}
\item[\rm{(T1)}] $T$ has projective dimension $\leq n$.
\item[\rm{(T2)}] $\mbox{Ext}^i_R(T,T^{(\kappa)}) = 0$ for all $1 \leq i \leq n$ and all cardinals $\kappa$.
\item[\rm{(T3)}] There exists an exact sequence $0 \to R \to T_0  \to \dots \to T_n \to 0$ where $T_i \in \Add (T)$ for each $i \leq n$.
\end{itemize}

\noindent Here, $\Add (T)$ denotes the class of all direct summands of (not necessarily finite) direct sums of copies of the module $T$. The class $\mathcal B = \{T\}^{\perp_\infty}$ is called the \emph{right $n$-tilting class}, $\mathcal A = {}^{\perp} \mathcal B$ the \emph{left $n$-tilting class}, and the cotorsion pair $(\mathcal A, \mathcal B )$ the \emph{$n$-tilting cotorsion pair} induced by $T$. Moreover, the kernel of this cotorsion pair, $\mathcal A \cap \mathcal B$ equals $\Add (T)$. 

If $T$ and $\tilde T$ are $n$-tilting modules, then $T$ is \emph{equivalent} to $\tilde T$, if  $\{ T \}^{\perp_\infty} = \{ \tilde T \}^{\perp_\infty}$, or equivalently, $\Add (T) = \Add (\tilde T)$. 
\end{definition}
 
The key fact about right $n$-tilting classes is that they are of \emph{finite type}, that is, $\mathcal B = \mathcal S ^{\perp}$ where $\mathcal S$ is a representative set of all strongly finitely presented modules in $\mathcal A$. The latter identity implies that $\mathcal B$ is a definable class, i.e.\ $\mathcal B$ is closed under products, direct limits and pure submodules. Moreover, $T$ and all its syzygies are direct summands of $\mathcal S$-filtered modules, and each module in $\mathcal S$ is a direct summand in a module filtered by $R$, $T$, the first syzygy of $T$, \dots, and the $(n-1)$-th syzygy of $T$, cf.\ \cite[6.14]{GT}. Also, each module $M \in \mathcal A$ is $\mathcal C$-filtered where $\mathcal C$ denotes the class of all strongly countably presented modules from $\mathcal A$, and $\mathcal A \subseteq \varinjlim \mathcal S$. For more details, we refer to \cite[\S13]{GT}. 
  
\medskip
The module theoretic notions above induce the corresponding global notions for quasi-coherent sheaves on schemes as follows:   

\begin{definition}\label{def-tiltqc} Let $n < \omega$ and $\mathcal Q$ be a quasi-coherent sheaf on a scheme $X$. Denote by $\mathcal O_{\!X}$ the structure sheaf of $X$. 

Then $\mathcal Q$ is \emph{locally left $n$-tilting} (\emph{locally kernel $n$-tilting}, \emph{locally $n$-tilting}) provided that for each open affine subset $U$ in $X$ there is an $n$-tilting $\mathcal O_{\!X}\!(U)$-module $T(U)$ such that the $\mathcal O_{\!X}\!(U)$-module of sections $\mathcal Q (U)$ satisfies $\mathcal Q (U) \in \mathcal A _{T(U)}$ ($\mathcal Q (U) \in \Add (T(U))$, and $\mathcal Q (U) = T(U)$, respectively).
\end{definition}

Since $0$-tilting modules coincide with projective generators, a quasi-coherent sheaf $\mathcal Q$ is locally left $0$-tilting, if and only if $\mathcal Q$ is locally kernel $0$-tilting, if and only if $\mathcal Q$ is an (infinite dimensional) vector bundle. In this sense, the notions of locally left $n$-tilting and locally kernel $n$-tilting quasi-coherent sheaves generalize the notion of a vector bundle. Similarly, a quasi-coherent sheaf $\mathcal Q$ is locally $0$-tilting, if and only if $\mathcal Q$ is a vector bundle such that $\mathcal Q (U)$ is a generator of $\rmod {\mathcal O_{\!X}\!(U)}$ for each open affine subset $U$ in $X$.      

\begin{remark} We prefer to use the adjective {\lq}locally{\rq} in Definition \ref{def-tiltqc} to avoid confusing these quasi-coherent sheaves with $n$-tilting objects of the Grothendieck category $\Qcoh{X}$ of all quasi-coherent sheaves on $X$. The latter objects are defined by the analogs of conditions (T1)-(T3) in $\Qcoh{X}$; in particular, they have no non-trivial self-extensions in $\Qcoh{X}$. In contrast, there exist vector bundles over non-affine schemes possessing non-trivial self-extensions. This can apply even to the structure sheaf, e.g.\ if $X$ is a smooth projective curve of genus $g>0$ over a field $k$, then $\Ext 1{\Qcoh{X}}{\mathcal O_{\!X}}{\mathcal O_{\!X}} \cong k^g \ne 0$.
\end{remark} 

The global notion for quasi-coherent sheaves corresponding to a property of $R$-modules $\mathfrak P _R$ is said to be \emph{Zariski local} in case the following holds true: If $X$ is a scheme with the structure sheaf $\mathcal O_{\!X}$, $X = \bigcup_{i \in I} \Spec {A_i}$ is an open affine covering of $X$, and $\mathcal Q$ is a quasi-coherent sheaf on $X$ such that the $A_i$-module of sections $\mathcal Q (\Spec {A_i})$ satisfies $\mathfrak P _{A_i}$ for each $i \in I$, then the $\mathcal O_{\!X}\!(U)$-module of sections $\mathcal Q (U)$ is satisfies $\mathfrak P _{O_{\!X}\!(U)}$ for all open affine subsets $U$ of $X$ (see \cite[\S 3]{EGT} and \cite[\S 5.3]{V}). 

\medskip
As mentioned above, our goal is to prove that the notions of left, kernel, and $n$-tilting quasi-coherent sheaves are Zariski local for all schemes $X$. We start with recalling the well-known relations of the Ext and Tor functors to flat change of rings: 

\begin{lemma}\label{localize} Let $\varphi : R \to S$ be a flat ring homomorphism of commutative rings, and $i < \omega$. 
\begin{enumerate}
\item Assume $A \in \rmod R$ is strongly finitely presented and $B \in \rmod R$. Then there is an isomorphism
$\Ext iRAB \otimes _R S \cong \Ext iS{A\otimes_R S}{B \otimes_R S}.$  
\item For all $A, B \in \rmod R$, there is an isomorphism $\Tor iRAB \otimes _R S \cong \Tor iS{A\otimes_R S}{B \otimes_R S}$.
\item If $A \in \rmod R$ and $B \in \rmod S$, then there are isomorphisms $\Ext iS{A\otimes_R S}B \cong \Ext iRAB$, and
$\Tor iS{A\otimes_R S}B \cong \Tor iRAB$.
\end{enumerate}
\end{lemma}  
\begin{proof} (1) follows e.g.\ by \cite[3.2.5]{EJ}, (2) by \cite[2.1.11]{EJ}, and (3) by \cite[VII.\S4]{CE}.
\end{proof}

\section{Ascent, descent, and tilting}
\label{sec:ad-tilting}

First, we recall that in order to prove Zariski locality, it suffices to check the validity of the assumptions of the following {\lq}Affine Communication Lemma{\rq} \cite[5.3.2.]{V} for the given setting (cf.\ \cite[3.5]{EGT}).  

\begin{lemma}\label{acl} Let $R$ be a commutative ring, $M \in \rmod R$, and $\mathfrak P _R$ be a property of $R$-modules such that
\begin{enumerate} 
\item if $M$ satisfies property $\mathfrak P _R$, then $M[f^{-1}] = M \otimes _R R[f^{-1}]$ satisfies property $\mathfrak P _{R[f^{-1}]}$ for each $f \in R$. 
\item if $R = \sum_{j < m} f_jR$, and the $R[f_j^{-1}]$-modules $M[f_j^{-1}] = M \otimes_R R[f_j^{-1}]$ satisfy property $\mathfrak P _{R[f_j^{-1}]}$ for all $j < m$, then $M$ satisfies property $\mathfrak P _R$.
\end{enumerate}
Then the global notion for quasi-coherent sheaves corresponding to $\mathfrak P _R$ is Zariski local.
\end{lemma}

In other words, the Affine Communication Lemma says that the global notion is Zariski local provided that the property $\mathfrak P _R$ \emph{ascends} along the flat ring epimorphisms $\varphi_f\colon R \to R[f^{-1}]$ for all $f \in R$, and \emph{descends} along the faithfully flat monomorphism $\varphi_{f_0,...,f_{m-1}}\colon R \to \prod_{i < m} R[f_i^{-1}]$ whenever $R = \sum_{j < m} f_jR$. 

If $\mathfrak P _R$ ascends along all flat ring homomorphisms, and descends along all faithfully flat ring homomorphism, then $\mathfrak P _R$ is called an \emph{ad-property} (cf.\ \cite[3.4]{EGT}).

We record several general important properties of faithfully flat ring homomorphisms (for the proof, see e.g.\ \cite[\S 7]{M}):

\begin{lemma}\label{faithful} Let $\varphi\colon R \to S$ be a faithfully flat ring homomorphism of commutative rings. Then $\varphi$ is monic (and hence, w.l.o.g., $\varphi$ is an inclusion). The induced map $\varphi^*\colon \Spec S \to \Spec R$ defined by $\q \mapsto \q \cap R$ is surjective. Each ideal $I$ of $R$ generates an ideal $IS \cong I \otimes _R S$ in $S$, and $IS \cap R = I$.

For each $\p \in \Spec R$, the poset $S_\p = \{ \q \in \Spec S \mid \q \cap R = \p \} \subseteq \Spec S$ is isomorphic to $\Spec{S \otimes _R \kappa(\p)}$, where $\kappa(\p)$ is the residue field of $\p$. The isomorphism takes $\q \in \Spec{S \otimes _R \kappa(\p)}$ to $\{ s \in S \mid s \otimes 1 \in \q \}$.     

If $\p \subseteq \p^\prime \in \Spec R$ and $\q^\prime \in \Spec S$ are such that $\q^\prime \cap R = \p^\prime$, then there exists $\q \in \Spec S$ such that $\q \subseteq \q^\prime$ and $\q \cap R = \p$.    
\end{lemma}

\medskip  
Now, we turn to the tilting setting. We will first establish a general lemma on the ascent and descent properties for suitable cotorsion pairs.

\begin{proposition}\label{ascends} Let $\varphi\colon R \to S$ be a flat ring homomorphism of commutative rings and $(\mathcal A, \mathcal B)$ be a cotorsion pair in $\ModR$ induced by a set $\mathcal S$ of strongly finitely presented modules (i.e.\ $\mathcal S ^{\perp} = \mathcal B$). Let us further denote by $(\mathcal A ^\prime, \mathcal B ^\prime)$ the cotorsion pair in $\ModS$ such that $\mathcal B ^\prime = (\mathcal S \otimes_R S)^{\perp}$. Then:

\begin{enumerate}
\item We have $\mathcal B ^\prime = \varphi_*^{-1}(\mathcal B)$, where $\varphi_*\colon\ModS\to\ModR$ is the forgetful functor. In particular, $\mathcal B \otimes _R S \subseteq \mathcal B ^\prime$.
If $\varphi$ is a faithfully flat ring homomorphism, then for each module $N \in \rmod R$, $N \in \mathcal B$, if and only if $N \otimes _R S \in \mathcal B ^\prime$.  
\item We have $\mathcal A \otimes _R S \subseteq \mathcal A ^\prime$. If $\varphi$ is a faithfully flat ring homomorphism, then for each $M \in \rmod R$, $M \in \mathcal A$, if and only if $M \otimes _R S \in \mathcal A ^\prime$. 
\end{enumerate}  
\end{proposition}          
\begin{proof} 
(1) By Lemma \ref{localize}(3), for each $B \in \rmod S$, $\Ext 1S{\mathcal S \otimes_R S}B \cong \Ext 1R{\mathcal S}B$, so $B \in \mathcal B ^\prime$, if and only if $\varphi_*(B) \in \mathcal B$. Hence 
$\mathcal B ^\prime = \varphi_*^{-1}(\mathcal B)$. 

In particular, since $S$ is a flat $R$-module, we have $\mathcal B \otimes _R S \subseteq \varinjlim \mathcal B \subseteq \mathcal B$ in $\rmod R$, so $\mathcal B \otimes _R S \subseteq \mathcal B ^\prime$ in $\rmod S$ by the above. 

If $\varphi$ is a faithfully flat ring homomorphism and $N \in \rmod R$, then Lemma \ref{localize}(1) shows that the condition $N \otimes _R S \in \mathcal B ^\prime$ is equivalent to $\Ext iR{\mathcal S}N = 0$ for each $0 < i < \omega$, that is, to $N \in \mathcal B$.  

(2) The class $\mathcal A$ (respectively, $\mathcal A ^\prime$), coincides with the class of all direct summands of $R$-modules ($S$-modules) filtered by the elements of $\mathcal S$ ($\mathcal S \otimes _RS$). Hence $\mathcal A ^\prime$ contains $\mathcal A \otimes _R S$.

Suppose now that $\varphi$ is faithfully flat and let $M \in \rmod R$ be such that $M^\prime = M \otimes _R S \in \mathcal A ^\prime$. Let $\{ m_\alpha \mid \alpha < \lambda \}$ be an $R$-generating subset of $M$, and $\mathcal M$ be a $\mathcal C ^\prime$-filtration of the module $M^\prime$, where $\mathcal C ^\prime$ denotes the class of all strongly countably presented modules from $\mathcal A ^\prime$. Let $\kappa = \aleph_1$ and consider a Hill family $\mathcal F$ of submodules of $M^\prime$ extending $\mathcal M$. 

By induction on $\alpha \leq \lambda$, we define a filtration $( M_\alpha \mid \alpha \leq \lambda )$ of $M$ such that $M_\alpha^\prime = M_\alpha \otimes _R S \in \mathcal F$, $m_\alpha \in M_{\alpha +1}$, $C_\alpha = M_{\alpha + 1}/M_\alpha$ is countably generated, and 
$C_\alpha ^\prime = M_{\alpha + 1}^\prime/M_\alpha^\prime \in \mathcal C ^\prime$ for each $\alpha < \lambda$. 

Let $M_0 = 0$, and $M_\alpha = \bigcup_{\beta < \alpha } M_\beta$ in case $\alpha \leq \lambda$ is a limit ordinal. If $M_\alpha$ has already been defined for some $\alpha < \lambda$, we let $N_0 = M_\alpha ^\prime \in \mathcal F$ and $P_0 = M_\alpha + m_\alpha R$. We take $N_1 \in \mathcal F$ such that $P_0 \otimes _R S \subseteq N_1$ and $N_1/N_0 \in \mathcal C ^\prime$ - this is possible because $\mathcal F$ is a Hill family. There is also $P_0 \subseteq P_1 \subseteq M$ such that $P_1/P_0$ is countably generated and $N_1 \subseteq P_1 \otimes _R S$. Then we take $N_2 \in \mathcal F$ such that $P_1 \otimes _R S \subseteq N_2$ and $N_2/N_1 \in \mathcal C ^\prime$. Proceeding similarly, we obtain two chains: $( P_i \mid i < \omega )$ of submodules of $M$ whose consecutive factors are countably generated, and  $( N_i \mid i < \omega )$ in $\mathcal F$ whose consecutive factors are in $\mathcal C ^\prime$, such that $P_i \otimes _R S \subseteq N_{i+1}$ and $N_i \subseteq P_i \otimes _R S$ for each $i < \omega$. Let $M_{\alpha + 1} = \bigcup_{i < \omega} P_i$. Then $m_\alpha \in M_{\alpha +1}$, $M_{\alpha +1}^\prime = \bigcup_{i < \omega} N_i \in \mathcal F$, $C_\alpha = M_{\alpha + 1}/M_\alpha$ is countably generated, and $C_\alpha ^\prime = M_{\alpha + 1}^\prime/M_\alpha^\prime \in \mathcal C ^\prime$.  

By construction, $C_\alpha ^\prime \in \mathcal C ^\prime$. Since $\varphi$ is faithfully flat, $C_\alpha$ is a strongly countably presented module by \cite[10.82.2]{SP}. Also $\mathcal C ^\prime \subseteq \mathcal A ^\prime \subseteq \varinjlim (\mathcal S \otimes _R S) = {}^\intercal ((\mathcal S \otimes _R S)^\intercal)$ by \cite[8.40]{GT}, whence Lemma \ref{localize}(2) and the faithful flatness of $\varphi$ yield $C_\alpha \in {}^\intercal (\mathcal S ^\intercal) = \varinjlim \mathcal S$.
  
So $C_\alpha = \varinjlim_{i \in I} S_i$ where $S_i \in \mathcal S$ for each $i \in I$. Then $C_\alpha ^\prime = \varinjlim_{i \in I} S_i \otimes _R S$. By \cite[9.2(3)]{AH}, $C_\alpha^\prime$ is $\mathcal B ^\prime$-stationary. In view of (2), this implies that for each $B \in \mathcal B$, the inverse system $\Hom S{S_i \otimes _R S}{B \otimes _R S}$ satisfies the Mittag-Leffler condition. The canonical isomorphism 
$\Hom S{S_i \otimes _R S}{B \otimes _R S} \cong \Hom S{S_i}B \otimes _R S$ from Lemma \ref{localize}(1) and the fact that $\varphi$ is faithfully flat yield that the inverse system $\Hom S{S_i}B$ satisfies the Mittag-Leffler condition. In other words, $C_\alpha$ is $\mathcal B$-stationary. Since $C_\alpha = \varinjlim_{i \in I} S_i$ and $C_\alpha$ is countably presented, \cite[9.2(3)]{AH} yields that $C_\alpha \in \mathcal A$. By the Eklof Lemma \cite[6.2]{GT}, also $M \in \mathcal A$.  
\end{proof}                      

If $T\in\ModR$ is an $n$-tilting module and $(\mathcal A, \mathcal B)$ is the associated cotorsion pair, then the above proposition applies. Indeed, the finite type of tilting classes implies that $\mathcal B$ is of the form $\mathcal S ^\perp$, where $\mathcal S$ is a representative set of all strongly finitely presented modules in $\mathcal{A}$, and the following lemma, which, in particular, refines some results known for the classical localization of tilting modules (cf.\ \cite[\S13.3]{GT}), shows that $(\mathcal A ^\prime, \mathcal B ^\prime)$ is an $n$-tilting cotorsion pair in $\ModS$.

\begin{lemma} \label{tilting ascent}
Let $\varphi\colon R \to S$ be a flat ring homomorphism of commutative rings and $T\in\ModR$ be an $n$-tilting module. Let $(\mathcal A,\mathcal B)$ be the induced $n$-tilting cotorsion pair in $\ModR$ and $\mathcal S$ be a representative set of the strongly finitely presented modules in $\mathcal{A}$.

\begin{enumerate}
\item The $S$-module $T^\prime = T \otimes _R S$ is $n$-tilting.
\item The class $\mathcal B ^\prime = (\mathcal S \otimes_R S)^{\perp}$ is the right tilting class for $T'$.
\end{enumerate}
\end{lemma}
\begin{proof}
(1) Applying the exact functor $- \otimes _R S$ to a projective resolution of $T$ of length $\leq n$ in $\rmod R$, we obtain a projective resolution of $T^\prime$ in $\rmod S$ of length $\leq n$, so condition (T1) holds for $T^\prime$. Similarly, applying $- \otimes _R S$ to the exact sequence from condition (T3) for $T$, we obtain condition (T3) for $T^\prime$.

By Lemma \ref{localize}(1), for each cardinal $\kappa$, we have $\Ext 1R{\mathcal S}{T^{(\kappa)}} \otimes _R S \cong \Ext 1S{\mathcal S \otimes _R S}{T^{(\kappa)} \otimes _R S}$, whence $\Ext 1S{\mathcal S \otimes _R S}{(T^\prime)^{(\kappa)}} = 0$ by the finite type of $T$. However, $T^\prime$ and all its syzygies are direct summands of $\mathcal S \otimes _R S$-filtered modules, so for each $0 < i \leq n$, $\Ext iS{T^\prime}{(T^\prime)^{(\kappa)}} = 0$, and condition (T2) holds for $T^\prime$. This proves that $T^\prime$ is an $n$-tilting $S$-module.

(2) $\Ext iS{T^\prime}N = 0$ for each $0 < i < \omega$, if and only if $\Ext 1S{\mathcal S \otimes _R S}N = 0$, so $\mathcal B ^\prime = (\mathcal S \otimes _R S) ^{\perp}$ is the right $n$-tilting class induced by $T^\prime$. 
\end{proof}

For the particular case of the $0$-tilting module $T = R$, Proposition~\ref{ascends}(2) shows that the property of being a projective module descends along arbitrary faithfully flat ring homomorphisms. Since projectivity obviously ascends along any flat homomorphism, it is an ad-property. Proposition \ref{ascends} thus implies the classic result of Raynaud and Gruson \cite{RG}: 

\begin{corollary} The property of being a projective module is an ad-property. In particular, the notion of an (infinite dimensional) vector bundle is Zariski local for all schemes.   
\end{corollary}

Moreover, we have

\begin{corollary}\label{acl1} Let $R$ be a commutative ring and $n < \omega$. Then the properties of being an element of a left $n$-tilting class, of the kernel of an $n$-tilting cotorsion pair, and being an $n$-tilting module, ascend along flat ring homomorphisms.
\end{corollary} 
\begin{proof} In view of Lemma~\ref{tilting ascent}(2), the first claim follows from part (2) of Proposition~\ref{ascends}, the second from its parts (1) and (2), and the third from part (1) of Lemma~\ref{tilting ascent}.
\end{proof}

\section{Tilting over commutative rings}
\label{sec:commutative-tilting}

We now turn to the structure of right tilting classes over commutative rings. These classes were recently described in terms of the spectrum of the ring \cite{APST,HS}. For this purpose, it is useful to consider a topology on the spectrum which is dual to the usual Zariski topology in the sense of Hochster. The base of closed sets for this topology consists of those sets which are open and quasi-compact in the Zariski topology. The open sets of this dual topology on the spectrum were used in Thomason's generalization of the Neeman-Hopkins classification of thick subcategories of the derived category of perfect complexes, which explains the following terminology:

\begin{definition}\label{defthom} Let $R$ be a commutative ring. A subset $P$ of $\mbox{Spec}(R)$ is called \emph{Thomason} (or \emph{Thomason open}) provided that $P$ is a union of Zariski closed sets with quasi-compact complements. Equivalently, there is a collection $\mathcal{I}$ of finitely generated ideals of $R$ such that $P = \bigcup_{I \in \mathcal{I}} V(I)$. We say that a Thomason set $P$ is \emph{basic}, if $P=V(I)$ for some finitely generated ideal $I$.
\end{definition}

The tilting classes over commutative rings are parametrized by finite filtrations of the spectrum by Thomason sets, satisfying an extra ``grade'' condition. We will call such filtrations characteristic sequences. We first express the ``grade'' condition in terms of avoiding primes associated to minimal cosyzygies of the regular module in a certain way, and then rephrase this condition homologically.

\begin{definition}\label{defvagass} Let $R$ be a commutative ring and $M$ an $R$-module. We say that a prime ideal $\p$ of $R$ is \emph{vaguely associated} to $M$ if $R/\p$ is contained in the smallest subclass of $\rmod R$ containing $M$ and closed under submodules and direct limits. Let $\mbox{VAss}{M}$ denote the set of all primes vaguely associated to $M$.
\end{definition}

\begin{definition}\label{defchar} Let $R$ be a commutative ring. A sequence $\bar P = (P_0,\dots,P_{n-1})$ consisting of Thomason subsets of $\Spec R$ is called \emph{characteristic} (of length $n$) provided that 
		\begin{enumerate}
				\item[(i)] $P_0 \supseteq P_1 \supseteq \dots \supseteq P_{n-1}$,
				\item[(ii)] $P_i \cap \VAss{\Omega^{-i}R} = \emptyset$ for each $i<n$, 
		\end{enumerate}
		where $\Omega^{-i}R$ denotes the $i$-th minimal cosyzygy of $R$.
\end{definition}

The condition $(ii)$ of Definition~\ref{defchar} can be reformulated in terms of Koszul cohomology, or its stable version, which will be crucial in our application. We briefly recall the relevant concepts. Given an element $x$ of a commutative ring $R$, we define the \emph{Koszul complex} $K_\bullet(x)$ with respect to $x$ as a complex
	$$\cdots \rightarrow 0 \rightarrow R \xrightarrow{\cdot x} R \rightarrow 0 \rightarrow \cdots$$
	concentrated in homological degrees 1 and 0. Given a sequence $x_1,x_2,\ldots,x_n$ of elements $R$, we define $K_\bullet(x_1,x_2,\ldots,x_n)$ to be the tensor product of complexes $K_\bullet(x_1) \otimes_R K_\bullet(x_2) \otimes_R \cdots \otimes_R K_\bullet(x_n)$. If $I$ is a finitely generated ideal of $R$, we fix once for ever a set of generators $x_1,x_2,\ldots,x_n \in R$ of $I$. For a module $M$, we define the $i$-th \emph{Koszul cohomology} of the ideal $I$ with respect to $M$ as follows:
	$$H^i_R(I;M)=H^i\Hom{R}{K_\bullet(x_1,x_2,\ldots,x_n)}{M}.$$
	We need to alert immediately that we abuse the notation to some extent. The Koszul cohomology is not invariant under the choice of generators of $I$. However, the vanishing of the cohomology up to any degree does not depend on this choice (see \cite[Proposition 3.4]{HS}), which justifies our use of this notation.
	
The issue with the ambiguity of the Koszul complex can be also mended at the cost of stepping outside of perfect complexes. We define the \emph{\v{C}ech complex} (also called the \emph{stable Koszul complex}) $\check{C}^\bullet(x)$ with respect to an element $x \in R$ to be the cochain complex
	$$\cdots \rightarrow 0 \rightarrow R \xrightarrow{\iota} R[x^{-1}] \rightarrow 0 \rightarrow \cdots,$$
	concentrated in cohomological degrees 0 and 1, where $R[x^{-1}]$ is the localization of the ring $R$ at the element $x$, and $\iota$ is the natural map. The term `stable Koszul complex' comes from the fact that
\[ \check{C}^\bullet(x_i) \cong \varinjlim\nolimits_{n\ge 1} \Hom{R}{K_\bullet(x^n)}{R} \cong \varinjlim\nolimits_{n\ge 1} K_\bullet(x^n)[-1]. \]
Given a sequence $x_1,x_2,\ldots,x_n$, we define the \v{C}ech complex $\check{C}^\bullet(x_1,x_2,\ldots,x_n)=\bigotimes_{i=1}^n \check{C}^\bullet(x_i)$. If $x_1,x_2,\ldots,x_n$ and $y_1,y_2,\ldots,y_m$ are two sets of generators of the same ideal $I$, the two associated \v{C}ech complexes are quasi-isomorphic (see \cite[Corollary 3.12]{Gr}). This justifies the following notation---given a finitely generated ideal $I$, we let $\check{C}^\bullet(I)$ be the \v{C}ech complex on some chosen finite generating set of $I$. Then $\check{C}^\bullet(I)$ is a well-defined object of the derived category of $R$. The \emph{\v{C}ech cohomology} of the ideal $I$ with respect to a module $M$ is then well-defined as follows:
	$$\check{H}^i_R(I;M)=H^i(\check{C}^\bullet(I) \otimes_R M).$$
In fact, for commutative noetherian rings $R$, or more generally for weakly proregular finitely generated ideals over arbitrary commutative rings, the \v{C}ech cohomology coincides with the local cohomology at $I$ (see~\cite[Theorem 3.2]{Sch}).

	Condition (ii) in Definition~\ref{defchar} can now be restated in various homological terms.
	
\begin{lemma}\emph(\cite[Theorem 3.14 and Lemma 7.4]{HS})\label{gradelemma}
		Let $R$ be a commutative ring, $P$ a Thomason subset of $\Spec{R}$, and $M$ an $R$-module. Let $\mathcal{I}$ be a set of finitely generated ideals such that $P=\bigcup_{I \in \mathcal{I}}V(I)$. Then the following conditions are equivalent for $n > 0$:
		\begin{enumerate}
				\item $P \cap \VAss{\Omega^{-i}M} = \emptyset$ for all $i=0,1,\ldots,n-1$,
				\item $\Ext{i}{R}{R/I}{M}=0$ for all $I \in \mathcal{I}$ and $i=0,1,\ldots,n-1$,
				\item $H^i_R(I;M)=0$ for all $I \in \mathcal{I}$ and $i=0,1,\ldots,n-1$,
				\item $\check{H}^i_R(I;M)=0$ for all $I \in \mathcal{I}$ and $i=0,1,\ldots,n-1$,
		\end{enumerate}
\end{lemma}
	One advantage of using Koszul cohomology (or its stable version) is its good behavior with respect to flat base change (unlike $\Ext{i}{R}{R/I}{-}$, which need not commute with direct limits for $i>0$).
\begin{lemma}\label{koszul-faith}
	Let $R$ be a commutative ring and $\varphi\colon R \rightarrow S$ a flat ring homomorphism. For any finitely generated ideal $I$ and any $R$-module $M$, we have $H_S^i(IS;M \otimes_R S)\cong H_R^i(I;M)\otimes_RS$. If, in particular, $\varphi$ is faithfully flat, we have $H_R^i(I;M)=0$ if and only if $H_S^i(IS;M \otimes_R S)=0$ for any $i$.
\end{lemma}
\begin{proof}
	Given a generating set $x_1,x_2,\ldots,x_n$ of $I$, it is easy to see that the complex $K_\bullet(x_1,x_2,\ldots,x_n) \otimes_R S$ is isomorphic to the Koszul complex over the ring $S$ on the generators $x_1 \otimes_R 1,x_2 \otimes_R 1,\ldots,x_n \otimes_R 1$. Since taking cohomology commutes with exact functors and $K_\bullet(I)$ is a perfect complex, we infer that 
		$$H^i\Hom{R}{K_\bullet(x_1,x_2,\ldots,x_n)}{M} \otimes_R S \cong H^i\Hom{S}{K_\bullet(x_1,x_2,\ldots,x_n) \otimes_R S}{M \otimes_R S}.$$
		Using our slightly abused notation, we can write $H_R^i(I;M) \otimes_R S \cong H_S^i(IS;M\otimes_RS)$. If $S$ is faithfully flat, the equivalence of vanishing of the cohomologies in the statement follows at once.
\end{proof}
 
The parametrization of right $n$-tilting classes over commutative rings by characteristic sequences was proved for the noetherian setting in \cite{APST}, and then generalized in \cite{HS}. Here, we use a version of the characterization employing the Tor functor, because this functor behaves very well w.r.t.\ flat base change (see Lemma \ref{localize}):  

\begin{theorem}\label{comnoe}\emph{\cite[Theorem 6.2]{HS}} Let $R$ be a commutative ring and $n < \omega$. 
		The right $n$-tilting classes in $\rmod R$ are parametrized by characteristic sequences of length $n$ in $\Spec R$: the class $\mathcal T _{\bar P}$ corresponding to a characteristic sequence $\bar P = (P_0,\dots,P_{n-1})$, where $P_i=\bigcup_{I \in \mathcal{I}_i}V(I)$ for a collection of finitely generated ideals $\mathcal{I}_i$ for each $i<n$, is defined by the formula
				$$\mathcal T _{\bar P} = \{ M \in \rmod R \mid \Tor iRM{R/I} = 0 \mbox{ for all } i < n \mbox{ and } I \in \mathcal{I}_i \}.$$
\end{theorem} 
	Using Koszul complexes, we can also compute the representative set $\mathcal{S}$ of strongly finitely presented modules associated to the tilting class. Let $I$ be a finitely generated ideal with a fixed generating set $x_1,x_2,\ldots,x_n$. Given any $i>0$, we denote by $S_{I,i}$ the cokernel of the map $\Hom{R}{d_i}{R}$, where
		$$\cdots \rightarrow F_n \xrightarrow{d_n}  F_{n-1} \xrightarrow{d_{n-1}} \cdots \xrightarrow{d_2} F_1 \xrightarrow{d_1} F_0 \rightarrow 0$$
		is the Koszul complex $K_\bullet(x_1,x_2,\ldots,x_n)$. Note that the module $S_{I,i}$ is a strongly finitely presented module of projective dimension $i$ whenever $H^j_R(I;R)=0$ for all $0\leq j<i$ or, equivalently by Lemma~\ref{gradelemma}, $\Ext{j}{R}{R/I}{R}$ for all $0\leq j<i$ (see also \cite[Proposition 5.12]{HS}).
		
\begin{lemma}\label{resolving}\emph{\cite[Theorem 6.2]{HS}}
		Let $\mathcal{T}_{\bar P}$ be the tilting class corresponding to the characteristic sequence $\bar{P}=(P_0,P_1,\ldots,P_{n-1})$ in the sense of Theorem~\ref{comnoe}. Then $\mathcal{T}_{\bar P}=\mathcal{S}^\perp$, where 
		$$\mathcal{S}=\{S_{I,i+1} \mid i<n, \text{ $I$ a finitely generated ideal such that } V(I) \subseteq P_i\}.$$
\end{lemma}
	Theorem \ref{comnoe} suggests an investigation of characteristic sequences under the base changes induced by (faithfully) flat ring homomorphisms:

\begin{lemma} \label{partic2} Let $R$ be a commutative ring and $n < \omega$. Let $\varphi: R \to S$ be a flat ring morphism and $\varphi^*\colon \Spec{S}\to\Spec{R}$ the induced map of the spectra. For each characteristic sequence of length $n$ in $\Spec R$, $\bar P = (P_0,\dots,P_{n-1})$, let $\bar Q _{\bar P} = (Q_0,\dots,Q_{n-1})$, where $Q_i = (\varphi^*)^{-1}(P_i) = \{ \q \in \Spec S \mid (\exists \p \in P_i)(\p S \subseteq \q) \}$ for each $i < n$.

\begin{enumerate}
\item The sequence $\bar Q _{\bar P}$ is characteristic of length $n$ in $\Spec S$.
\item If $T$ is an $n$-tilting module inducing the right $n$-tilting class $\mathcal B = \mathcal T _{\bar P}$ in $\rmod R$, then $T^\prime = T \otimes _R S$ is an $n$-tilting $S$-module inducing the right $n$-tilting class $\mathcal B ^\prime = \mathcal T _{\bar Q _{\bar P}}$ in $\rmod S$. 
\end{enumerate}

If, moreover, $\varphi$ is faithfully flat, then the assignment $\bar P \mapsto \bar Q _{\bar P}$ is monic.
\end{lemma}

\begin{proof} Throughout the proof, we denote by $\mathcal{I}_i$ a set of finitely generated ideals of $R$ such that $P_i = \bigcup_{I \in \mathcal{I}_i}V(I)$.
		
		(i) We start by proving that $\bar Q _{\bar P}$ is a characteristic sequence in $\Spec S$. First, it is clear that $P_i \supseteq P_{i+1}$ implies $Q_i \supseteq Q_{i+1}$ for each $i<n-1$. Since $(\varphi^{-1})(V(I)) = V(SI)$ for each ideal $I\subseteq R$, it follows that $Q_i = \bigcup_{I \in \mathcal{I}_i} V(IS)$.
	As $IS$ is finitely generated for any $I \in \mathcal{I}_i$, $Q_i$ is a Thomason subset of $\mbox{Spec}(S)$ for any $i<n$ (in other words, $\varphi^*$ is continuous with respect to the Thomason topologies).

		Fix an ideal $I \in \mathcal{I}_i$. Since $\bar P$ is characteristic, we have $H_R^j(I;R)=0$ for all $j \leq i$. Then Lemma~\ref{koszul-faith} yields $H_S^j(IS;S)=0$ for all $j \leq i$. This shows that $\bar Q$ is a characteristic sequence in $\Spec S$.

		(ii) By Lemma~\ref{resolving}, the $n$-tilting class $\mathcal{B}$ equals $\bigcap_{i=0}^{n-1}\bigcap_{I \in \mathcal{I}_i} (S_{I,i+1})^\perp$, where $S_{I,i+1}$ is the cokernel of the map $\Hom{R}{d_{i+1}}{R}$, where $d_{i+1}$ is the map in degree $i+1$ of the Koszul complex $K_\bullet(x_1,\ldots,x_n)$ on generators $\{x_1,\ldots,x_n\}$ of ideal $I$. Similarly to Lemma~\ref{koszul-faith}, it is easy to see that $S_{I,i+1} \otimes_R S$ is the cokernel of the map $\Hom{S}{d'_{i+1}}{S}$, where
		$$F_n \otimes_R S \xrightarrow{d'_n}  F_{n-1} \otimes_R S \xrightarrow{d'_{n-1}} \cdots \xrightarrow{d'_2} F_1 \otimes_R S \xrightarrow{d'_1} F_0 \otimes_R S \rightarrow 0$$
		is (isomorphic to) the Koszul complex $K_\bullet(x_1 \otimes_R 1, \ldots, x_n \otimes_R 1)$ on the corresponding generators of the ideal $IS$ of ring $S$. By Proposition~\ref{ascends}, we have $\mathcal{B}'= \bigcap_{i=0}^{n-1}\bigcap_{I \in \mathcal{I}_i} (S_{I,i+1} \otimes_R S)^\perp$. By using Lemma~\ref{resolving} again, the latter class equals $\mathcal{T}_{\bar{Q}_{\bar{P}}}$.

		Finally, in order to show that the assignment $\bar{P} \mapsto \bar{Q}_{\bar{P}}$ is monic in the faithfully flat situation, it is enough to show that $\varphi^*[Q_i]=\{\q \cap R \mid \q \in Q_i\}=P_i$ for each $i<n$. The inclusion $\varphi^*[Q_i] \subseteq P_i$ has been already proved in the first part. Let now $\p \in P_i$ and $I \in \mathcal{I}_i$ be such that $I \subseteq \p$. By Lemma~\ref{faithful}, there is $\q \in \Spec{S}$ such that $\varphi^*(\q)=\q\cap R=\p$. Then we have 
		$$IS\subseteq \p S = (\q \cap R)S \subseteq \q,$$
		proving that $\q \in Q_i$.
\end{proof}

Next we show how the characteristic sequence $\bar{P}$ can be recovered from the $n$-tilting module $T$ in homological terms. In order to do this, we will need to recall the dual setting (we refer to \cite[Chapter 15]{GT} for details on cotilting modules and duality).

Let $R$ be a commutative ring and $T$ an $n$-tilting module corresponding to a characteristic sequence $\bar{P}=(P_0,P_1,\ldots,P_{n-1})$ in the sense of Theorem~\ref{comnoe}. We fix an injective cogenerator $W$ of $\ModR$ and let $(-)^+=\Hom{R}{-}{W}$ denote the duality with respect to $W$. Then the dual module $C=T^+$ is an \emph{$n$-cotilting module}, that is, an $R$-module satisfying the conditions dual to Definition~\ref{tilt}: 
\begin{definition}\label{cotilt} Let $R$ be a ring and $n<\omega$. A left $R$-module $C$ is \emph{n-cotilting} provided that:
\begin{itemize}
\item[\rm{(C1)}] $C$ has injective dimension $\leq n$.
\item[\rm{(C2)}] $\mbox{Ext}^i_R(C^\kappa,C) = 0$ for all $1 \leq i \leq n$ and all cardinals $\kappa$.
\item[\rm{(C3)}] There exists an exact sequence $0 \to C_n \to \dots \to C_0 \to W \to 0$ where $W$ is an injective cogenerator of $\ModR$, where $C_i \in \Prod (C)$ for each $i \leq n$, and where $\Prod(C)$ denotes the class of direct summands of direct products of copies of $C$.
\end{itemize}
\end{definition}
The class $\mathcal{C}={}^{\perp_\infty}\{C\}$ is the \emph{$n$-cotilting class} associated to $C$. Then the $n$-cotilting class $\mathcal{C}$ arising in this way is uniquely determined by the tilting class $T^\perp$; we say that $\mathcal{C}$ corresponds to the characteristic sequence $\bar{P}$. We gather several useful properties of cotilting classes dual to tilting classes from \cite{APST} and \cite{HS}.
Given a class of modules $\mathcal D$, we denote by $\Ass{}{\mathcal{D}}$ is the union of the sets of associated prime ideals taken over all $M\in\mathcal{D}$.

\begin{lemma}\label{cotilting}
		Let $R$ be a commutative ring, $T$ an $n$-tilting module with the corresponding characteristic sequence $\bar{P}=(P_0,P_1,\ldots,P_{n-1})$, $C=T^+$, and $\mathcal{C}={}^{\perp_\infty}\{C\}$ the corresponding $n$-cotilting class. Then:
		\begin{enumerate}
				\item $\mathcal{C}=\bigcap_{i=0}^{n-1} \bigcap_{I\in\mathcal I _i} \Ker{\check{H}^i(I;-)}$, where $\mathcal I _i$ is a set of finitely generated ideals such that $\bigcup_{I\in \mathcal I _i} V(I) = P_i$ (cf. Lemma~\ref{gradelemma}).
				\item For any $i<n$, the class $\mathcal{C}_{(i)}={}^{\perp_\infty}(\Omega^{-i}(C))$ is an $(n-i)$-cotilting class corresponding to the characteristic sequence $(P_i,P_{i+1},\ldots,P_{n-1})$.
				\item $P_i = \Spec{R} \setminus \Ass{}{\mathcal{C}_{(i)}}$ for any $i<n$.
		\end{enumerate}
\end{lemma}
\begin{proof}
		\begin{enumerate}
			\item This is by \cite[Theorem 7.7]{HS}.
			\item This follows from \cite[Lemma 3.5]{APST} and \cite[Lemma 5.14]{HS}.
			\item We refer to \cite[Theorem 6.1]{HS}.
		\end{enumerate}
\end{proof}

	Let $R$ be a commutative ring. We recall that a Thomason set $P$ induces a torsion pair $(\mathcal{T}(P),\mathcal{F}(P))$, where $\mathcal{F}(P)=\{M \in \ModR \mid \VAss{M} \cap P = \emptyset\}$. This is a \emph{hereditary torsion pair of finite type} (i.e., a torsion pair such that the torsion-free class is closed under taking injective envelopes and direct limits), and the assignment $P \mapsto (\mathcal{T}(P),\mathcal{F}(P))$ induces a bijection between Thomason sets in $\Spec{R}$ and hereditary torsion pairs of finite type in $\ModR$ (see e.g.\ \cite[Propositions 2.11 and 2.13]{HS}).

Recall that for a prime $\p \in \Spec{R}$, we denote by $\kappa(\p)$ the quotient field of $R/\p$.
\begin{lemma}\label{Cech_residue}
	Let $I$ be a finitely generated ideal and $\p \in \Spec{R}$. Then 

		$$\check{C}(I) \otimes_R \kappa(\p) \text{ is quasi-isomorphic to } \begin{cases}\kappa(\p), & \text{if } \p \in V(I) \\ 0, & \text{if } \p \not\in V(I). \end{cases}$$
\end{lemma}
\begin{proof}
		It is easy to see that $\check{C}(I) \otimes_R \kappa(\p)$ is isomorphic to the \v{C}ech complex over the ring $\kappa(\p)$ with respect to the ideal $\bar{I}$---the image of $I$ in the natural map $R \rightarrow \kappa(\p)$. If $I \subseteq \p$, then $\bar{I}=0$, and the complex $\check{C}(I) \otimes_R \kappa(\p)$ is quasi-isomorphic to $\kappa(\p)$. If $I \not\subseteq \p$, then $\bar{I}=\kappa(\p)$, and $\check{C}(I) \otimes_R \kappa(\p)$ is exact.
\end{proof}

Now we can obtain the following description of cotilting classes corresponding to characteristic sequences and their relation to residue fields.

\begin{lemma}\label{cotilting2}
		Let $R$ be a commutative ring, $n>0$, $\bar{P}=(P_0,P_1,\ldots,P_{n-1})$ a characteristic sequence, and $\mathcal{C}$ an $n$-cotilting class corresponding to $\bar{P}$. Then:
		\begin{enumerate}
				  \item $\mathcal{C}=\{M \in \ModR \mid \Omega^{-i}(M) \in \mathcal{F}(P_i) \text{ for all $i<n$}\}$.
		    	\item Given $\p \in \Spec{R}$, we have the following equivalences for each $i<n$:
					
					\smallskip
					\begin{enumerate}
					\item[$(\ast)$] $\p \in P_i \Leftrightarrow \kappa(\p) \in \mathcal{T}(P_i)$,
					\item[$(\ast\ast)$] $\p \in \Ass{}{\mathcal{C}_{(i)}} \Leftrightarrow \p \not\in P_i \Leftrightarrow \kappa(\p) \in \mathcal{F}(P_i) \Leftrightarrow \kappa(\p) \in \mathcal{C}_{(i)}$.
					\end{enumerate}
					\smallskip
					
			  	\item $\mathcal{F}(P_i)$ is equal to the closure of $\mathcal{C}_{(i)}$ under submodules. In particular, $\mathcal{F}(P_0)$ is cogenerated by $C$, whenever $C$ is an $n$-cotilting module such that $\mathcal{C}={}^{\perp_\infty}\{C\}$.
		\end{enumerate}
\end{lemma}
\begin{proof}
(i) This is \cite[Theorem 5.3]{HS}.

(ii) If $\p \in P_i$, then $\mathcal{F}(P_i) \subseteq \Ker{\Hom{R}{R/\p}{-}}$ (see e.g. \cite[Lemma 3.12]{HS}), and thus $R/\p \in \mathcal{T}(P_i)$. Since $\kappa(\p)$ is a flat $R/\p$-module, this yields $\kappa(\p) \in \mathcal{T}(P_i)$. If $\p \not\in P_i$, then $\kappa(\p) \in \mathcal{F}(P_i)$, since $\VAss{(\kappa(\p))}=\{\p\}$. However, $(\mathcal{T}(P),\mathcal{F}(P))$ is a torsion pair, so $(\ast)$ follows, as well as the second equivalence in $(\ast\ast)$. The first equivalence in $(\ast\ast)$ holds by Lemma \ref{cotilting}(iii). It remains to prove that 
		$\p \not\in P_i$, if and only if $\kappa(\p) \in \mathcal{C}_{(i)}$. 
				
		By Lemma~\ref{cotilting}(i) and (ii), an $R$-module $M$ belongs to $\mathcal{C}_{(i)}$ if and only if $\check{H}^{j-i}_R(I;M)=0$ for any $I\in\mathcal I_j$ and any $i \leq j <n$, where as usual $\mathcal I _j$ stands for a set of finitely generated ideals such that $P_j = \bigcup_{I \in \mathcal I _j} V(I)$. By Lemma~\ref{Cech_residue}, 
\begin{align*}
\check{H}^{j-i}_R(I;\kappa(\p)) = 0 \text{ for all } i \leq j <n \text{ and } I \in \mathcal I _j
&\iff \check{H}^0_R(I;\kappa(\p)) = 0 \textit{ for all } I \in \mathcal I _i \\
&\iff \p \not\in P_i,
\end{align*}
establishing the equivalence.

(iii) This follows from (the proof of) \cite[Lemma 5.7]{HS} together with \cite[Lemma 5.8]{HS}.
\end{proof}

In the setting of Lemma~\ref{cotilting}(ii), we have the following increasing chain of classes in $\rmod R$:
$$\mathcal{C}_{(-1)} := \{ 0 \} \subseteq \mathcal C = \mathcal{C}_{(0)} \subseteq \dots \subseteq \mathcal{C}_{(i-1)} \subseteq \mathcal{C}_{(i)} \subseteq \dots \subseteq \mathcal{C}_{(n)} = \rmod R.$$ 
We also have the decreasing chain of subsets of $\Spec R$:
$$P_{-1} := \Spec R \supseteq P_0  \supseteq \dots \supseteq P_{i-1} \supseteq P_{i} \supseteq \dots \supseteq P_{n-1} \supseteq P_n := \emptyset.$$
Of course, for each $\p \in \Spec R$, there is a unique $i \leq n$, such that $\p \in P_{i-1} \setminus P_i$, or equivalently (see Lemma \ref{cotilting2}(ii)), $\kappa(\p) \in \mathcal{C}_{(i)} \setminus \mathcal{C}_{(i-1)}$. The index $i$ can be determined simply by checking vanishing of the groups $\Tor{j}{R}{\kappa(\p)}{T}$ (for $j \leq n$):

\begin{proposition}\label{characteristic2} 
For each $j < \omega$, we have $\Tor{j}{R}{\kappa(\p)}{T} = 0$, if and only if $j \neq i$. In particular, $P_i = \{ \p \in P_{i-1} \mid \Tor{i}{R}{\kappa(\p)}{T}=0\}.$
\end{proposition}
\begin{proof}
Since $\kappa(\p) \in \mathcal{C}_{(i)} \setminus \mathcal{C}_{(i-1)}$ and $C=T^+$, we in particular have
\[
\kappa(\p) \in \bigcap_{j>i}\Ker{\Ext{j}{R}{-}{T^+}} = \bigcap_{j>i}\Ker{\Tor{j}{R}{-}{T}}.
\]
Thus $\Tor{j}{R}{\kappa(\p)}{T}=0$ for all $j>i$. Since $\kappa(\p)\not\in\mathcal{C}_{(i-1)}$, we have by the same token $\Tor{i}{R}{\kappa(\p)}{T} \neq 0$. 

For each $j<i$, Lemma~\ref{cotilting2}(ii) yields $\kappa(\p) \in \mathcal{T}(P_j)$. This implies $\Hom{R}{\kappa(\p)}{C}=0$, i.e., $\kappa(\p) \otimes _R T = 0$. For $0<j<i$, apply $\Hom{R}{\kappa(\p)}{-}$ to the exact sequence $0 \rightarrow \Omega^{-(j-1)}(C) \rightarrow E(\Omega^{-(j-1)}C)\rightarrow \Omega^{-j}(C) \rightarrow 0$. Since both $\Omega^{-(j-1)}(C)$ and its injective envelope belong to $\mathcal{F}(P_{j-1})$ by Lemma~\ref{cotilting2}(i), we get $\Hom{R}{\kappa(\p)}{\Omega^{-j}C} \cong \Ext{1}{R}{\kappa(\p)}{\Omega^{-(j-1)}(C)} \cong \Ext{j}{R}{\kappa(\p)}{C} \cong (\Tor{j}{R}{\kappa(\p)}{T})^+$. Since $\Omega^{-j}(C) \in \mathcal{F}(P_j)$ and $\kappa(\p) \in \mathcal{T}(P_j)$, we conclude that $\Tor{j}{R}{\kappa(\p)}{T} = 0$.	
\end{proof}

As already mentioned, the Thomason sets of the spectrum are precisely the open sets of the Hochster dual of the spectrum, \cite{H}. This justifies the following definition---a subset $X$ of $\Spec{R}$ is \emph{Thomason closed} if its complement $\Spec{R} \setminus X$ is Thomason. It is easy to see from the definition that Thomason closed subsets are precisely (arbitrary) intersections of quasi-compact Zariski open subsets.

\begin{lemma}\label{Tho-max}
	Let $X_i$, $i \in I$, be a collection of Thomason closed subsets of $\Spec{R}$. Then whenever we chose $\q_i \in X_i$ for each $i \in I$, there is a maximal (with respect to inclusion) element $\q \in \bigcap_{i \in I}X_i$ such that $\bigcap_{i \in I} \q_i \subseteq \q$.

		In particular, any Thomason closed set $X$ has enough maximal elements (i.e.\ each $\q\in X$ is contained in $\q'$ which is maximal in $X$ with respect to inclusion).
\end{lemma}
\begin{proof}
		Since $\bigcap_{i \in I}X_i$ is a Thomason closed set, its complement $Z=\Spec{R} \setminus \bigcap_{i \in I}X_i$ is Thomason, and thus $Z=\bigcup_{J \in \mathcal{J}}V(J)$ for some set $\mathcal{J}$ of finitely generated ideals of $J$. We can also assume without loss of generality that $\mathcal{J}$ is closed under ideal product, as replacing it by its closure under products does not alter the induced Zariski closed subset of the spectrum. Clearly, the ideal $\bigcap_{i \in I}\q_i$ does not contain $J$ for any $J \in \mathcal{J}$. Since the ideals of $\mathcal{J}$ are finitely generated, we can use the Zorn Lemma to find a maximal ideal $\q$ containing $\bigcap_{i \in I}\q_i$ such that $J \not\subseteq \q$ for any $J \in \mathcal{J}$. As $\mathcal{J}$ is closed under products, it is easy to check that $\q$ is a prime ideal.
\end{proof}

Now we can establish one of the key lemmas in our attempt to prove descent for tilting classes.

\begin{lemma}\label{Tho-loc}
		Let $\varphi: R \rightarrow S$ be a faithfully flat ring homomorphism. Then $\varphi$ is Thomason closed, i.e., given any Thomason closed subset $X \subseteq \Spec{S}$, the image $\varphi^*[X]$ is Thomason closed in $\Spec{R}$.
\end{lemma}
\begin{proof}
		First, we handle the case when $X$ is a basic Thomason closed set, that is, let $X$ be an Zariski open and quasi-compact subset of $\Spec{S}$. By Lemma~\ref{faithful}, the image $Y=\varphi^*[X]$ is a lower subset of $\Spec{R}$ with respect to inclusion. Also, since $\varphi^*$ is Zariski continuous, $Y$ is Zariski quasi-compact. We show that $Y$ is Thomason closed. To this end, we let
\[ Y'=\bigcap\{U \subseteq \Spec{R} \mid U \text{ is Zariski open and quasi-compact, and } Y \subseteq U\}.\]
Clearly $Y'$ is Thomason closed, and $Y \subseteq Y'$. We prove $Y=Y'$.

		Let $\q \in \Spec{R} \setminus Y$. Because $Y$ is a lower set in $\Spec{R}$, it is an intersection of Zariski open sets. Therefore, there is an open set containing $Y$, but not $\q$. Because the quasi-compact open sets form an open base of the Zariski topology,
there is for any $\p \in Y$ a quasi-compact open set $O_\p$ such that $\p \in O_\p$, and $\q \not\in O_\p$. The collection $\{O_\p \mid \p \in Y\}$ forms an open covering of $Y$. By the quasi-compactness of $Y$, there are primes $\p_1,\ldots,\p_n$ such that $Y \subseteq O = \bigcup_{i=1}^n O_{\p_i}$. Then $O$ is a quasi-compact open set containing $Y$, but not $\q$. Therefore, $\q \not\in Y'$. This proves $Y=Y'$, and thus also that $Y$ is Thomason closed.

		Let now $X$ be any Thomason closed subset of $\Spec{S}$. Then $X$ is expressible as an intersection of quasi-compact Zariski open sets, say $X=\bigcap_{i \in I}X_i$. We let $Y=\varphi^*[X]$, and $Y_i=\varphi^*[X_i]$. The situation is as follows: $Y \subseteq \bigcap_{i \in I} Y_i$, and each $Y_i$ is Thomason closed by the first part of the proof. We prove that $Y=\bigcap_{i \in I} Y_i$.

		Since $Y$ is a lower set, and any element of $\bigcap_{i \in I} Y_i$ is contained in some of its maximal elements by Lemma~\ref{Tho-max}, it is enough to show that any maximal element $\p$ of $\bigcap_{i \in I} Y_i$ is contained in $Y$. Let us fix $\p$ and, for each $i \in I$, let $\q_i \in X_i$ be such that $\p=\varphi^*(\q_i)=\q_i \cap R$. By Lemma~\ref{Tho-max}, there is a maximal element $\q$ of $X$ such that $\bigcap_{i \in I}\q_i \subseteq \q$. Then we have
		$$\varphi^*(\q)=\q \cap R \supseteq \bigcap_{i \in I}\q_i \cap R = \p.$$
		On the other hand, $\varphi^*(\q) \in \bigcap_{i \in I}Y_i$, and $\p$ is a maximal element of this set. Therefore, $\varphi^*(\q)=\p$, proving that $\p \in Y$. This concludes the proof.
\end{proof}

\begin{proposition}\label{descends}
Let $R$ be a commutative ring and $n < \omega$. Let $\varphi\colon R \to S$ be a faithfully flat ring homomorphism. 

Let $T^\prime$ be an $n$-tilting $S$-module of the form  $T^\prime = \tilde T \otimes _R S$ for some module $\tilde T \in \rmod R$. Let $(\mathcal A ^\prime, \mathcal B ^\prime)$ be the corresponding $n$-tilting cotorsion pair in $\rmod S$.  
 
\begin{enumerate}
\item Let $\bar Q = (Q_0,\dots,Q_{n-1})$ be the characteristic sequence of length $n$ in $\Spec S$ corresponding to the right $n$-tilting class $\mathcal B ^\prime$. Then $\bar Q = \bar Q _{\bar P}$ (in the sense of Lemma~\ref{partic2}) for a characteristic sequence $\bar P = (P_0,\dots,P_{n-1})$ of length $n$ in $\Spec R$. In particular, there is an $n$-tilting module $T \in \rmod R$ such that $T \otimes _R S$ is equivalent to $T^\prime$. 
\item Let $(\mathcal A, \mathcal B)$ be the $n$-tilting cotorsion pair corresponding to the $n$-tilting $R$-module $T$ from part (1). Let $M \in \rmod R$ be such that $M^\prime = M \otimes _R S \in \mathcal A ^\prime$ ($M^\prime = M \otimes _R S \in \mathcal B ^\prime$, or $M^\prime \in \Add (T^\prime )$). Then $M \in \mathcal A$ ($M \in \mathcal B$, or $M \in \Add (T)$, respectively).             
\end{enumerate}
\end{proposition}          
\begin{proof} 
		(1) First, we prove that if $i < n$ and $\q \in Q_i$, then $\q^\prime \in Q_i$ for each $\q^\prime \in \Spec S$ such that $\q' \cap R = \q \cap R$ (i.e.\ that each $Q_i$ is a union of fibers of the map $\varphi^*\colon \Spec S \to \Spec R$). 	
		
		Let $\q \in \Spec{S}$ and denote $\p = \q \cap R \in \Spec{R}$. The inclusion $R/\p \xhookrightarrow{} S/\q$ induces a field extension $\kappa(\p) \xhookrightarrow{} \kappa(\q)$ of the quotient fields of the domains $R/\p$ and $S/\q$, respectively. 
		Now we can compute:
\begin{align*}
\Tor{i}{S}{\kappa(\q)}{T'} &\cong H_i(\kappa(\q) \otimes_S^\mathbb{L} T') \\
&\cong H_i(\kappa(\q) \otimes_S^\mathbb{L} (\tilde{T} \otimes_R S)) \\
&\cong H_i((\kappa(\q) \otimes_S S) \otimes_R^\mathbb{L} \tilde{T}) \\
&\cong H_i(\kappa(\q) \otimes_R^\mathbb{L} \tilde{T}) \\
&\cong H_i( (\kappa(\p) \otimes_{\kappa(\p)} \kappa(\q)) \otimes_R^\mathbb{L} \tilde{T}) \\
&\cong H_i( (\kappa(\p) \otimes_R^\mathbb{L} \tilde{T}) \otimes_{\kappa(\p)} \kappa(\q)) \\
&\cong H_i(\kappa(\p) \otimes_R^\mathbb{L} \tilde{T}) \otimes_{\kappa(\p)} \kappa(\q) \\
&\cong \Tor{i}{R}{\kappa(\p)}{\tilde{T}} \otimes_{\kappa(\p)} \kappa(\q).
\end{align*}
		Since field extensions are faithfully flat, we conclude that $\Tor{i}{S}{\kappa(\q)}{T'} =0$ if and only if $\Tor{i}{R}{\kappa(\p)}{\tilde{T}}=0$ for any $\q \in \Spec{S}$ such that $\p=\q \cap R$. From this, it is straightforward to deduce by induction on $i \geq 0$ and using Proposition~\ref{characteristic2} that given any $\q \in Q_i$, the condition $\q \cap R = \q' \cap R$ implies $\q' \in Q_i$ for any $\q' \in \Spec{S}$.

		Now we can prove that the sequence $(P_0,P_1,\ldots,P_{n-1})$, defined by setting $P_i=\{\q \cap R \mid \q \in Q_i\}$ for $i<n$, is characteristic. First we show that $P_i$ is a Thomason subset of $\Spec{R}$ for each $i<n$. By the previous observation, we see that $\varphi^*[\Spec{S} \setminus P_i] = \Spec{R} \setminus \varphi^*[P_i]$. Since $\varphi^*$ preserves Thomason closed sets by Lemma~\ref{Tho-loc}, we infer that $Q_i=\varphi^*[P_i]$ is a Thomason subset of $\Spec{R}$.

		If $\q' \in \Spec{S}$ is such that $\p S \subseteq \q'$ for some $\p \in P_i$, then $\q' \cap R \in P_i$, implying $\q' \in Q_i$. This shows that $\bar{Q}_{\bar{P}}=\bar{Q}$. Now we can check conditions (i) and (ii) from Definition~\ref{defchar}, the first one being obvious. To show the second one, let $\mathcal{I}_i$ be a collection of finitely generated ideals such that $P_i=\bigcup_{I \in \mathcal{I}_i}V(I)$, provided by $P_i$ being a Thomason set. Since $\bar{Q}_{\bar{P}}=\bar{Q}$, the proof of Lemma~\ref{partic2}(1) shows that $Q_i=\bigcup_{I \in \mathcal{I}_i}V(SI)$. Now we apply Lemma~\ref{koszul-faith}.

		Finally, let $T$ be the $n$-tilting $R$-module corresponding to $\bar P$ by Theorem \ref{comnoe}. By Lemma \ref{partic2}(ii), $\bar Q$ is the characteristic sequence corresponding to the $n$-tilting $S$-module $T \otimes _R S$, so the latter module is equivalent to $T^\prime$, i.e.\ $\Add (T') = \Add (T\otimes_RS)$.

(2) The first case follows from part (2) of Proposition~\ref{ascends}, the second from its (1), and the third by both parts (1) and (2) and the fact that $\Add (T) = \mathcal A \cap \mathcal B$.
\end{proof}

As an immediate consequence, we obtain a counterpart of Corollary~\ref{acl1} for the descent. The descent for tilting modules will be treated in detail in the last section.

\begin{corollary} \label{main_lk} The properties of being a left $n$-tilting, and a kernel $n$-tilting module are ad-properties. In particular, the notions of a locally left $n$-tilting, and a locally kernel $n$-tilting quasi-coherent sheaf are Zariski local for all schemes.
\end{corollary}

\section{Descent of tilting modules}
\label{sec:descent}

Let us quickly sum up what we have proved so far. Let $\varphi: R \rightarrow S$ be a faithfully flat ring homomorphism and $\tilde{T}$ an $R$-module such that $T'=\tilde{T}\otimes_RS$ is an $n$-tilting $S$-module. Then Proposition~\ref{descends} shows that conditions (T1) and (T2) of Definition~\ref{tilt} are satisfied by the $R$-module $\tilde{T}$. It remains to prove the descent of condition of (T3) in the presence of (T1) and (T2). In this section, we show that this holds in the following cases:

\begin{enumerate}
		\item if $\varphi$ is a Zariski covering, that is, $\varphi$ is of the form as in Lemma \ref{acl}(ii) (see Theorem~\ref{main} below),
		\item if $T'$ is 1-tilting (Theorem~\ref{main-1dim}),
		\item if $R$ is noetherian (Corollary~\ref{main-noetherian}), and
		\item if the $n$-tilting $S$-module $T'$ corresponds to a characteristic sequence consisting of basic Thomason sets (see Definition~\ref{basic} and Theorem~\ref{main-basic}).
\end{enumerate}

For this purpose, it is convenient to replace (T3) in the definition of a tilting module by a condition asserting that $T$ is a (weak) generator in the unbounded derived category $\mathbf{D}(R)$ of $R$-modules. It is proved in \cite[Corollary 2.6]{PS} that, given a module $T$ and assuming conditions (T1) and (T2), the condition (T3) can be equivalently replaced by.
\begin{enumerate}
		\item[\rm{(T3')}] For any $X \in \mathbf{D}(R)$, we have $\mathbb{R}\!\Hom{R}{T}{X}=0 \implies X=0$.
\end{enumerate}
	Dually, by \cite[Remark 3.7]{PS}, we can also replace the condition (C3) in the definition of an $n$-cotilting module $C$ by
\begin{enumerate}
		\item[\rm{(C3')}] For any $X \in \mathbf{D}(R)$, we have $\mathbb{R}\!\Hom{R}{X}{C}=0 \implies X=0$.
\end{enumerate}

\noindent
	We further recall that a triangulated subcategory of $\mathbf{D}(R)$ is \emph{localizing (resp. colocalizing)} if it is closed and arbitrary direct sums (resp. direct products). Given a class $\mathcal{C} \subseteq \mathbf{D}(R)$, we denote by $\Loc(\mathcal{C})$ (resp. $\Coloc(\mathcal{C})$) the smallest localizing (resp. colocalizing) subcategory containing $\mathcal{C}$. Note that, using the properties of $\mathbb{R}\!\Hom{R}{-}{X}$, condition (T3') is implied by $\Loc(T)=\mathbf{D}(R)$. On the other hand, condition (T3) implies that $R \in \Loc(T)$, showing that the following condition can be also equivalently used in place of (T3),
\begin{enumerate}
		\item[\rm{(T3'')}] $\Loc(T)=\mathbf{D}(R)$.
\end{enumerate}
Dually, we can replace the condition (C3) by
\begin{enumerate}
		\item[\rm{(C3'')}] $\Coloc(C)=\mathbf{D}(R)$.
\end{enumerate}

Using \v{C}ech complexes, we check (T3'') directly in the case when the faithfully flat ring homomorphism comes from a Zariski open covering of $\Spec R$, i.e.\ it is of the form as in the paragraph after Lemma \ref{acl}:
\begin{lemma}\label{T3}
Let $R$ be a commutative ring. Let $R = \sum_{i < m} f_iR$ and
$$ \varphi = \varphi_{f_0,...,f_{m-1}}\colon R \to S = \prod_{i < m} R[f_i^{-1}] $$
be the faithfully flat ring morphism, where the components $R\to R[f_i^{-1}]$ are the localization morphisms.
		Let $\tilde{T}$ be an $R$-module satisfying conditions (T1) and (T2) of Definition \ref{tilt} such that the $S$-module $\tilde{T}\otimes _R S$ is $n$-tilting. Then $\tilde T$ is an $n$-tilting $R$-module.
\end{lemma}
\begin{proof} 
		We consider the \v{C}ech complex $\check{C}_\bullet(f_0,f_1,\ldots,f_{m-1})$. Explicitly, this complex has the following form:
		$$R \rightarrow \bigoplus_{j<m}R[f_j^{-1}] \rightarrow \bigoplus_{j<j'<m}R[f_j^{-1},f_{j'}^{-1}] \rightarrow \cdots \rightarrow R[f_0^{-1},f_1^{-1},\ldots,f_{m-1}^{-1}].$$
		Since the quasi-isomorphism class of a \v{C}ech complex does only depend on the ideal generated by the elements $f_0,f_1,\ldots,f_{m-1}$ (see e.g. \cite[Corollary 3.12]{Gr}), $\check{C}_\bullet(f_0,f_1,\ldots,f_{m-1})$ is quasi-isomorphic to the exact complex $\check{C}_\bullet(1) = (0 \rightarrow R \overset{\cong}\rightarrow R \rightarrow 0)$, and thus is exact.

		Now consider the truncated \v{C}ech complex $\check{C}'_\bullet(f_0,f_1,\ldots,f_{m-1})$:
		$$\bigoplus_{j<m}R[f_j^{-1}] \rightarrow \bigoplus_{j<j'<m}R[f_j^{-1},f_{j'}^{-1}] \rightarrow \cdots \rightarrow R[f_0^{-1},f_1^{-1},\ldots,f_{m-1}^{-1}].$$
		Using the exactness of $\check{C}_\bullet(f_0,f_1,\ldots,f_{m-1})$, the map $R \rightarrow \bigoplus_{j<m}R[f_j^{-1}]$ induces a quasi-isomorphism $R \cong \check{C}'_\bullet(f_0,f_1,\ldots,f_{m-1})$. 

		The proof is concluded by showing that $\check{C}'_\bullet(f_0,f_1,\ldots,f_{m-1}) \in \Loc(\tilde T)$. By \cite[Lemma 1.1.8]{KP}, any localizing subcategory is a tensor ideal, i.e.\ it is closed under tensoring by any object of $\mathbf{D}(R)$. Therefore, $\Loc(\tilde T \otimes_R S) \subseteq \Loc(\tilde T)$. Since $\tilde T \otimes_R S$ is a tilting $S$-module, we have $S \in \Loc(\tilde T \otimes_R S)$, and thus $R[f_j^{-1}] \in \Loc(\tilde T)$ for any $j<m$. Then also $R[f_{j_1}^{-1},f_{j_2}^{-1},\ldots,f_{j_k}^{-1}] \in \Loc(\tilde T)$ for any $j_1 < j_2 < \cdots <j_k <m$, proving finally that $R \cong \check{C}'_\bullet(f_0,f_1,\ldots,f_{m-1}) \in \Loc(\tilde T)$, and therefore $\Loc(\tilde T) = \mathbf{D}(R)$.
\end{proof}

Together with Corollary \ref{acl1} and Proposition~\ref{descends}, this yields the desired result on Zariski locality for arbitrary schemes. 

\begin{theorem} \label{main} The notion of locally $n$-tilting quasi-coherent sheaf is Zariski local for all schemes.
\end{theorem}

Using Neeman's classification of localizing subcategories of the derived category of a noetherian commutative ring, we can check (T3'') quite easily in the case when $R$ is noetherian. Given a complex $X_\bullet \in \mathbf{D}(R)$, we denote its \emph{cohomological support} by $\supph{X_\bullet} = \{ \p \in \Spec{R} \mid \kappa(\p) \otimes_R^{\mathbb{L}} X_\bullet \neq 0\}$. If $M$ is a module, then $\supph{M}$ is the cohomological support of the stalk complex $M_\bullet$, that is, $\supph(M) = \{\p \in \Spec{R} \mid (\exists i\ge 0)(\Tor{i}{R}{\kappa(\p)}{M} \neq 0)\}$. Neeman's \cite[Theorem 2.8]{N} now says that each localizing subcategory of $\mathbf{D}(R)$ is of the form $\{X \in \mathbf{D}(R) \mid \supph{X} \subseteq P\}$, where $P$ is a subset $\Spec{R}$.

\begin{proposition}\label{kanon_na_vrabce}
	Let $R$ be a commutative noetherian ring, and $\varphi\colon R \rightarrow S$ be a faithfully flat ring homomorphism. Let $\tilde T$ be an $R$-module satisfying conditions (T1) and (T2) of Definition \ref{tilt} such that the $S$-module $\tilde T\otimes _R S$ is $n$-tilting. Then $\tilde T$ is an $n$-tilting $R$-module.
\end{proposition}
\begin{proof}
		As in the proof of Lemma~\ref{T3}, it suffices to show that $R \in \Loc(S)$. By Neeman's theorem, there is a subset $P \subseteq \Spec{R}$ such that $\Loc(S) = \{X \in \mathbf{D}(R) \mid \supph{X} \subseteq P\}$. Note that $\supph{S}=\Spec{R}$. Indeed, as $S$ is flat, $\kappa(\p) \otimes_R^{\mathbb{L}} S \cong \kappa(\p) \otimes_R S$, and since $S$ is faithful, this tensor is non-zero for any $\p \in \Spec{R}$. But since $S \in \Loc(S)$, we have $P=\Spec{R}$, and thus $\Loc(S) = \mathbf{D}(R)$.
\end{proof}

\begin{corollary} \label{main-noetherian} The notion of an $n$-tilting module descends along faithfully flat homomorphisms $\varphi\colon R \rightarrow S$ of commutative rings such that $R$ is noetherian.
\end{corollary}

Along the way, we have collected some additional information on modules $\tilde{T}\in\ModR$ which ascend over a faithfully flat homomorphism $\varphi\colon R\to S$ to an $n$-tilting module $T':=\tilde{T}\otimes_RS$. From Proposition~\ref{descends} and Corollary~\ref{main_lk} we already knew that there existed an $n$-tilting module $T\in\ModR$ such that $\tilde{T}\in\Add{T}$ and $\Add(T')=\Add(T\otimes_RS)$. From the proof of Proposition~\ref{kanon_na_vrabce}, we now know that $\supph{\tilde{T}} = \Spec{R}$. In the case $n=0$, $\tilde{T}$ is a projective module with a full support, so necessarily a projective generator (e.g. because its trace ideal must be the whole ring $R$). In the general case, we can say the following:

\begin{lemma}\label{Tor-progenerator}
		Let $T$ be an $n$-tilting $R$-module corresponding to a characteristic sequence $\bar{P}=(P_0,P_1,\ldots,P_{n-1})$. If $\tilde{T} \in \Add(T)$ is such that $\supph{\tilde{T}} = \Spec{R}$, and $I$ is a finitely generated ideal such that $V(I) \subseteq P_{n-1}$, then
		$$\Tor{i}{R}{\tilde{T}}{R/I}=\begin{cases}0, & i \neq n, \\ \text{a projective generator in $\rmod{R/I}$}, & i=n. \end{cases}$$
In other words, $\tilde{T} \otimes_R^\mathbb{L} R/I$ is quasi-isomorphic to $P[-n]$, where $P$ is a projective generator of $\rmod{R/I}$.
\end{lemma}
\begin{proof}
		By Theorem~\ref{comnoe}, $\Tor{i}{R}{\tilde{T}}{R/I}=0$ for any $i \neq n$, because $\tilde{T} \in \Add(T)$. It is enough to show that $\Tor{i}{R}{\tilde{T}}{R/I}$ is a projective generator of $\operatorname{Mod-R/I}$. Let
		$$0 \rightarrow P_n \rightarrow P_{n-1} \rightarrow \cdots \rightarrow P_0 \rightarrow T \rightarrow 0$$
		be a projective resolution of $T$. Applying $- \otimes_R R/I$ yields an exact sequence
		\begin{equation}\label{splitexact}
		P_n \otimes_R R/I \rightarrow P_{n-1} \otimes_R R/I \rightarrow \cdots \rightarrow P_0 \otimes_R R/I \rightarrow 0,
		\end{equation}
		and since $P_i \otimes_R R/I$ is a projective $R/I$-module for each $i \leq n$, the sequence (\ref{splitexact}) is split. Therefore, the kernel of the leftmost map of (\ref{splitexact}) is a projective $R/I$-module. But this kernel is precisely $\Tor{n}{R}{\tilde{T}}{R/I}$. Finally, for any $\p \in \Spec{R/I} = V(I)$, applying $- \otimes_R \kappa(\p)$ onto the split exact complex (\ref{splitexact}) yields:
		$$\Tor{n}{R}{\tilde{T}}{R/I} \otimes_{R/I} \kappa(\p) \cong \Tor{n}{R}{\tilde{T}}{\kappa(\p)} \neq 0,$$
where the last inequality follows from the assumption that $\supph{\tilde{T}}=\Spec R$.
This proves that the projective $R/I$-module $\Tor{n}{R}{\tilde{T}}{R/I}$ is a generator of $\rmod{R/I}$.
\end{proof}

Recall from Definition~\ref{defthom} that a Thomason set $P$ is basic if it is of the form $V(I)$, where $I$ is a finitely generated ideal. This is equivalent to saying that there is a finite collection $I_1,I_2,\ldots,I_m$ of finitely generated ideals such that $P=\bigcup_{i=1}^m V(I_i)$, as then $P=V(I_1I_2\ldots I_m)$. A key fact about basic Thomason sets is that they are quasi-compact with respect to the Thomason topology. Although this follows from the discussion in~\cite{KP}, we give a short and elementary proof for the convenience of the reader.

\begin{lemma}\label{Thomason qc}
		Let $R$ be a commutative ring, $J$ be an ideal and $\mathcal{I}$ an arbitrary collection of finitely generated ideals. If $V(J) = \bigcup_{I\in\mathcal{I}} V(I)$ in $\Spec{R}$, then there exists a finite subcollection $\mathcal{I}'\subseteq\mathcal{I}$ such that $V(J) = \bigcup_{I\in\mathcal{I}'} V(I)$.
\end{lemma}

\begin{proof}
Let $(\mathcal{T},\mathcal{F})$ be the torsion pair in $\ModR$ with $\mathcal{F}=\bigcap_{I\in\mathcal{I}}\Ker{\Hom{R}{R/I}{-}}$. The modules $M\in\mathcal{T}$ are precisely those for which each $x\in M$ is annihilated by the product $I_1\cdots I_m$ of (not necessarily distinct) ideals $I_1,\dots,I_m\in \mathcal{I}$ (see \cite[Lemma 2.3]{Hrb}). We also denote by $\mathcal{K}$ the collection of proper ideals $K\subsetneqq R$ such that $R/K\in\mathcal{F}$. Since $\mathcal{F}$ is closed under direct limits, $\mathcal{K}$ is closed under unions of chains and,  by the Zorn Lemma, each ideal in $\mathcal{K}$ is contained in a maximal one with respect to inclusion. Moreover, if $\p$ is maximal in $\mathcal{K}$, it is a prime ideal. Indeed, if $\p = K_1K_2$ and $\p\subsetneqq K_1$, then $0\ne K_1/\p\in\mathcal{F}$ is annihilated by $K_2$, so $K_2=\p$ by the maximality of $\p$ (cp.~\cite[Lemma 3.9]{Hrb}).

Now we claim that $R/J \in \mathcal{T}$. Otherwise there would exist $K\in\mathcal{K}$ with $K\supseteq J$, and hence also $\p\in V(J)$ with $R/\p\in\mathcal{F}$. However, we have $\p\in V(I)$ for some $I\in\mathcal{I}$ by assumption, which implies $R/\p\in\mathcal{T}$, a contradiction. Therefore, by the above description of modules in $\mathcal{T}$, there exist ideals $I_1,\dots,I_m\in \mathcal{I}$ such that the product $I_1\cdots I_m$ annihilates $R/J$. This implies that $V(J) = V(I_1)\cup\cdots\cup V(I_m)$.
\end{proof}

Now we can extend the definition of basic Thomason sets to characteristic sequences and prove the descent of tilting modules for them.

\begin{definition} \label{basic}
We call a characteristic sequence $(P_0,P_1,\ldots,P_{n-1})$ \emph{basic} if each of the Thomason sets $P_i$, $0 \leq i < n$, is basic.
\end{definition}

\begin{theorem}\label{main-basic}
		Let $\varphi\colon R \rightarrow S$ be a faithfully flat ring homomorphism and $\tilde{T}$ an $R$-module such that $\tilde{T}\otimes_RS$ is an $n$-tilting $S$-module corresponding to a basic characteristic sequence $\bar{Q}=(Q_0,Q_1,\ldots,Q_{n-1})$ of $\Spec{S}$. Then $\tilde{T}$ is an $n$-tilting module.
\end{theorem}
\begin{proof}
		Let $\bar{P}=(P_0,P_1,\ldots,P_{n-1})$ be the characteristic sequence obtained from $\bar{Q}$ by Proposition~\ref{descends}, that is, $Q_i=(\varphi^*)^{-1}(P_i)$ for $i<n$. Then $\bar{P}$ is basic as well.
Indeed, fix $i<n$ and let $J$ be a finitely generated ideal of $S$ such that $Q_i=V(J)$. Since $P_i$ is Thomason, there is a set $\mathcal{I}$ of finitely generated ideals of $R$ such that $P_i=\bigcup_{I \in \mathcal{I}}V(I)$. Then $Q_i = V(J) = \bigcup_{I \in \mathcal{I}}V(SI)$, so $Q_i = V(SI_1) \cup \dots \cup V(SI_m) = V(SI_1\cdots I_m)$ for some $I_1,\dots,I_m\in\mathcal{I}$ by Lemma~\ref{Thomason qc}. It follows from Lemma~\ref{faithful} that $P_i = V(SI_1\cdots I_m \cap R) = V(I_1\cdots I_m)$ is basic, as claimed.

		Let $(\mathcal{A},\mathcal{B})$ be the $n$-tilting cotorsion pair corresponding to $\bar{P}$ in the sense of Theorem~\ref{comnoe}. By Proposition~\ref{descends} and Proposition~\ref{characteristic2}, 
		$$\tilde{T} \in \mathcal{A} \cap \mathcal{B} \text{ and } \supph{\tilde{T}} = \Spec{R}.$$ 
		We aim to show by induction on $n \geq 0$ that $\Loc(\tilde{T})=\mathbf{D}(R)$. If $n=0$, then $\tilde{T}$ is a projective module such that $\tilde{T} \otimes_R \kappa(\p) \neq 0$ for any $\p \in \Spec{R}$. This implies that the trace ideal of $\tilde{T}$ is necessarily the whole $R$, and therefore $\tilde{T}$ is a projective generator, whence a 0-tilting module.
		
		Suppose now that the result holds for all dimensions smaller than $n$. Let $I$ be an ideal of $R$ generated by elements $x_1,x_2,\ldots,x_m$ such that $P_{n-1}=V(I)$. By Proposition~\ref{ascends} and Lemma~\ref{partic2}, the localization $\tilde{T}[x_i^{-1}]$ is a module belonging to $\mathcal{A}_i \cap \mathcal{B}_i$, where $(\mathcal{A}_i,\mathcal{B}_i)$ is an $n$-tilting cotorsion pair in $\rmod{R[x_i^{-1}]}$ corresponding to the basic characteristic sequence $(Q^i_0,Q^i_1,\ldots,Q^i_{n-1})$ given by $Q^i_j=P_j \cap \Spec{R[x_i^{-1}]}$. But since $x_i \in I$, we see that $Q^i_{n-1}=\emptyset$, and thus $(\mathcal{A}_i,\mathcal{B}_i)$ is actually an $(n-1)$-tilting pair. Clearly, 
		$$\tilde{T}[x_i^{-1}] \otimes_{R[x_i^{-1}]}^\mathbb{L} \kappa(\q) \cong \tilde{T} \otimes_R^\mathbb{L} ({R[x_i^{-1}]} \otimes_R \kappa(\q)) \cong \tilde{T} \otimes_R^\mathbb{L} \kappa(\q) \neq 0$$ 
		for any $\q \in \Spec{R[x_i^{-1}]}$, and thus $\supph_{R[x_i^{-1}]} \tilde{T}[x_i^{-1}] = \Spec{R[x_i^{-1}]}$. Therefore, our induction hypothesis applies and $\tilde{T}[x_i^{-1}]$ is an $(n-1)$-tilting $R[x_i^{-1}]$-module for all $i<m$. By condition (T3''), we have $R[x_i^{-1}] \in \Loc(\tilde{T}[x_i^{-1}]) \subseteq \Loc(\tilde{T})$ for all $i<m$, and thus we can argue as in the proof of Proposition~\ref{kanon_na_vrabce} that the truncated \v{C}ech complex $\check{C}'_\bullet(I)$ belongs to $\Loc(\tilde{T})$. 

		On the other hand, the (non-truncated) \v{C}ech complex $\check{C}_\bullet(I)$ belongs to $\Loc(R/I)$ (see e.g. \cite[\S 7.4]{HS}), and $\Loc(R/I) = \Loc(\tilde{T} \otimes_R^\mathbb{L} R/I) \subseteq \Loc(\tilde{T})$ by Lemma~\ref{Tor-progenerator}. Finally, the triangle
		$$\check{C}_\bullet(I) \rightarrow R \rightarrow \check{C}'_\bullet(I) \rightarrow \check{C}_\bullet(I)[1]$$
		shows that $R \in \Loc(\tilde{T})$.
\end{proof}

\begin{remark}
A part of the proof of Theorem~\ref{main-basic} would work even for a general characteristic sequence $\bar{P}=(P_0,P_1,\ldots,P_{n-1})$. The Thomason set $P_{n-1}$ induces a localizing subcategory of the derived category $\mathbf{D}(R)$, and there is a triangle
$$\Gamma(R) \rightarrow R \rightarrow L(R) \rightarrow \Gamma(R)[1],$$
with $\Gamma(R) \in \Loc(\{R/I \mid I\in\mathcal{I}_{n-1}\}$ and $L(R) \in \bigcap_{I\in\mathcal{I}_{n-1}}\Ker{\mathbb{R}\!\Hom{R}{R/I}{-}}$, where $\mathcal{I}_{n-1}$ is a set of finitely generated ideals such that $P_{n-1} = \bigcup_{I\in\mathcal{I}_{n-1}} V(I)$ (see e.g.\ \cite[Propositions 1.1.5 and 2.1.2]{KP}). It can be shown similarly as in the preceding proof that $\Gamma(R) \in \Loc(\tilde T)$, but it is not clear to us in general how to prove $L(R) \in \Loc(\tilde T)$ when $P_{n-1}$ is not basic.
\end{remark}

However, we can still prove the general faithfully flat descent in the case of tilting modules of projective dimension $\leq 1$, using different methods. For this, we need several preparatory steps.

In what follows, let $\varphi\colon R \to S$ be a faithfully flat homomorphism of commutative rings and $\varphi^*\colon \Spec S \to \Spec R$ the corresponding map between the spectra. Suppose that $\tilde{T}\in\ModR$ is an $R$-module such that $T' = \tilde{T}\otimes_RS$ is a $1$-tilting $S$-module corresponding to a Thomason set $Q = Q_0 \subseteq \Spec{S}$. Also let $P=\varphi^*[Q]$ be the Thomason set obtained by Proposition~\ref{descends}, and let $T$ be a 1-tilting $R$-module corresponding to $P$ in the sense of Theorem~\ref{comnoe}, and denote by $(\mathcal{A}, \mathcal{B})$ the tilting cotorsion pair in $\ModR$ corresponding to $T$, that is,
\begin{align*}
\mathcal{B} &= \{ B \in \ModR \mid BI=B \text{ for each  finitely generated } I \subseteq R \text{ such that } V(I) \subseteq P \} \\
&= \bigcap_{V(I)\subseteq P} \Ker{R/I \otimes_R -}.
\end{align*}

We now make a key observation about the character module $\tilde{T}^+$.

\begin{lemma} \label{lem:dual-cotilting}
In the setting as above, $\tilde{T}^+$ is a $1$-cotilting $R$-module and the corresponding cotilting class is
\[ \mathcal{F} = \bigcap_{V(I)\subseteq P} \Ker{\Hom R{R/I}{-}}. \]
\end{lemma}

\begin{proof}
We already know that $\tilde{T}^+ \in \Prod(T^+)$ and that $T^+$ is a $1$-cotilting module with the associated cotilting class as above. In particular, conditions (C1) and (C2) are satisfied for $\tilde{T}^+$. It suffices to prove (C3'') for $\tilde{T}^+$.

Since each localizing subcategory is a tensor ideal and $S$ is flat over $R$, we have $T' = \tilde{T}\otimes_RS \in \Loc(\tilde{T})$, when we consider $T'$ as an $R$-module. By the exact sequence (T3) for $T'$, we obtain $S \in \Loc(\tilde{T})$.

When passing to character modules and complexes, we see that the colocalizing subcategory $\Coloc(\tilde{T}^+) \subseteq \mathbf{D}(R)$ contains $S^+$, which is an injective cogenerator for $\ModR$ as $\varphi$ is faithfully flat. In particular, $\Coloc(\tilde{T}^+) = \mathbf{D}(R)$ and (C3'') holds for $\tilde{T}^+$.
\end{proof}

Note that the class $\mathcal{F}$ is the torsion-free class of the hereditary torsion pair of finite type $(\mathcal{T}, \mathcal{F})$ corresponding to the Thomason set $P$. We next observe that $\tilde{T}$ can be used to describe the torsion and torsion-free classes of this torsion theory as follows.

\begin{lemma} \label{lem:tensor-perp}
With the above notation, we have $\mathcal{T} = \Ker{\tilde{T}\otimes_R-}$ and $\mathcal{F} = \Ker{\Tor 1R{\tilde T}-}$. 
\end{lemma}

\begin{proof}
Thanks to Lemma~\ref{lem:dual-cotilting}, we have the formulas $\mathcal{T} = \Ker{\Hom R-{\tilde{T}^+}}$ and $\mathcal{F} = \Ker{\Ext 1R-{\tilde{T}^+}}$. The conclusion follows by a standard adjunction.
\end{proof}

Now we can prove that $\tilde{T}$ is a \emph{partial 1-tilting} $R$-module, which by definition means that $\tilde{T}$ satisfies (T1), (T2) and $\tilde{T}^{\perp}$ is a $1$-tilting class. 

\begin{proposition} \label{prop:partial-1-tilt}
In the notation as above, we have $\mathcal{B} = \tilde T ^\perp$.
\end{proposition}

\begin{proof}
Since $\mathcal{B} = \Ker{\Ext 1R{T}{-}}$ and $\tilde{T}\in\Add(T)$, we certainly have
\[ \mathcal{B} \subseteq \Ker{\Ext 1R{\tilde T}{-}}. \]

For the other inclusion, it suffices to show that each $M \in \ModR$ admits a short exact sequence of the form $0 \to M \to B \to \tilde{T}^{(J)} \to 0$ with $B \in \mathcal{B}$. Indeed, such a sequence splits for any $M \in \Ker{\Ext 1R{\tilde T}{-}}$, so any such $M$ lies in $\mathcal{B}$. Furthermore, since $\mathcal{B}$ is closed under direct sums and epimorphic images, it suffices to construct such an exact sequence only for $M = R$. Indeed, for general $M$ we can take an epimorphism $\pi\colon R^{(K)} \to M$ and a pushout of a direct sum of $K$ copies of $0 \to R \to B \to \tilde{T}^{(J)} \to 0$ along $\pi$.

To construct the sequence of $R$, we consider the so-called Bongartz completion, i.e.\ any short exact sequence
\[ \varepsilon\colon \quad 0 \longrightarrow R \longrightarrow B \longrightarrow \tilde{T}^{(J)} \longrightarrow 0 \]
such that the connecting homomorphism
\[ \delta\colon \Hom R{\tilde T}{\tilde{T}^{(J)}} \to \Ext 1R{\tilde T}R \]
is surjective.

Our task is to prove that $B \in \mathcal{B}$, or equivalently that $B\otimes_RR/I = 0$ for any fixed finitely generated ideal $I\subseteq R$ such that $V(I)\subseteq P$. If we apply $-\otimes_R R/I$ to $\varepsilon$, we get in view of Lemma~\ref{lem:tensor-perp} an exact sequence
\[ \Tor 1R{\tilde{T}^{(J)}}{R/I} \overset{\gamma_I}\longrightarrow R/I \longrightarrow B\otimes_RR/I \longrightarrow 0. \]
Hence, our task is equivalent to proving that the connecting homomorphism $\gamma_I$ is surjective.  

Observe, moreover, that it suffices to construct any extension $\varepsilon'\colon 0 \to R \to E \to \tilde{T}^{(K)}
 \to 0$ for which the connective homomorphism $\gamma'\colon \Tor 1R{\tilde{T}^{(K)}}{R/I} \to R \otimes_R R/I \cong R/I$ is surjective. Indeed, since $\varepsilon$ is a Bongartz completion and $\delta$ above is surjective, there exists a commutative diagram
\[
\begin{CD}
\varepsilon'\colon\quad @. 0 @>>> R @>>> E @>>> \tilde{T}^{(K)} @>>> 0 \\
@.                         @.    @|    @VVV     @VVV                   \\
\varepsilon\colon\quad  @. 0 @>>> R @>>> B @>>> \tilde{T}^{(J)} @>>> 0 \\
\end{CD}
\]
Using the naturality of the long exact sequences obtained by applying $-\otimes_R R/I$ to the rows of the diagram, it follows that $\gamma'$ factorizes through $\gamma_I$. In particular, $\gamma_I$ is surjective provided that $\gamma'$ is such.

In order to construct an extension $\varepsilon'$ with $\gamma'$ surjective, we start with an arbitrary (split) epimorphism of $R/I$-modules
\[ \rho\colon \Tor 1R{\tilde{T}^{(K)}}{R/I} \cong \Tor 1R{\tilde T}{R/I}^{(K)} \longrightarrow R/I. \]
Such an epimorphism exists by Lemma~\ref{Tor-progenerator}. We also fix a projective presentation
\[ 0 \longrightarrow P_1 \longrightarrow P_0 \longrightarrow \tilde{T}^{(K)} \to 0. \]
Since, as in the proof of Lemma~\ref{Tor-progenerator}, we have a split embedding $\Tor 1R{\tilde{T}^{(K)}}{R/I} \to P_1 \otimes_R R/I$, we can extend $\rho$ to a map of $R/I$-modules $\rho'\colon P_1 \otimes_R R/I \to R/I$ and compose it further with the projection $P_1 \cong P_1 \otimes_R R \to P_1 \otimes_R R/I$ to obtain a homomorphism of $R$-modules $P_1 \to R/I$. As $P_1$ is projective, the latter morphism factorizes as $P_1 \overset{\vartheta}\to R \to R/I$, where the second map is the canonical projection. The construction is summarized in the following diagram, where the lower row composes to $\rho$,
\[
\begin{CD}
                              @.                     P_1             @>{\vartheta}>> R  \\
@.                                                  @VVV                  @VVV          \\
\Tor 1R{\tilde{T}^{(K)}}{R/I} @>>{\subseteq_\oplus}> P_1 \otimes_R R/I @>>{\rho'}> R/I  \\
\end{CD}
\]

Finally, we construct $\varepsilon'$ as the lower row of the pushout diagram
\[
\begin{CD}
                        @. 0 @>>> P_1 @>>>  P_0 @>>> \tilde{T}^{(K)} @>>> 0 \\
@.                      @. @V{\vartheta}VV  @VVV     @|                     \\
\varepsilon'\colon\quad @. 0 @>>> R   @>>>  E   @>>> \tilde{T}^{(K)} @>>> 0 \\
\end{CD}
\]

It is now a standard computation that $\gamma'\colon \Tor 1R{\tilde{T}^{(K)}}{R/I} \to R/I$, the connecting homomorphism, coincides with the composition
\[ \Tor 1R{\tilde{T}^{(K)}}{R/I} \overset{\subseteq_\oplus}\longrightarrow P_1 \otimes_R R/I \overset{\vartheta\otimes_R R/I}\longrightarrow R/I \] 
and that $\vartheta\otimes_R R/I$ identifies with $\rho'$. It follows that $\gamma' = \rho$, and $\rho$ is surjective from the construction.
\end{proof}

Before finally proving the faithfully flat descend for 1-tilting modules, we need to introduce one last notion---the tilting ring epimorphism. These were introduced in \cite{CTT}, and here we will follow \cite[\S 4]{MS}. Suppose that $\tilde T$ is a partial 1-tilting module. Then $\mathcal{X}_{\tilde T}=\{M \in \ModR \mid \Hom{R}{\tilde T}{M}=0 ~\&~ \Ext{1}{R}{\tilde T}{M}=0\}$ is a \emph{bireflective category} (that is, a full subcategory closed under kernels, cokernels, products and coproducts), that is also closed under extensions. Any bireflective subcategory $\mathcal{X}$ gives rise to a ring epimorphism $\lambda_\mathcal{X}: R \rightarrow U$ such that $\mathcal{X}$ is precisely the essential image of the scalar restriction functor $\rmod{U} \rightarrow \ModR$. In this way, we assign to a partial tilting module $\tilde T$ a \emph{tilting ring epimorphism} $\lambda_{\tilde T}:=\lambda_{\mathcal{X}_{\tilde{T}}}$.

\begin{lemma}\label{lem:residue-duality}
	Let $M$ be an $R$-module and $\p$ a prime. Then for any $i \geq 0$ we have
		$$\Tor{i}{R}{M}{\kappa(\p)}=0 \iff \Ext{i}{R}{M}{\kappa(\p)}=0.$$
\end{lemma}
\begin{proof}
		Using the (derived) Hom-$\otimes$ adjunction, there is an isomorphism
		\begin{equation}\label{eq:residue-duality}
		\Ext{i}{R}{M}{\Hom{R}{\kappa(\p)}{E(R/\p)}} \cong \Hom{R}{\Tor{i}{R}{M}{\kappa(\p)}}{E(R/\p)}.
		\end{equation}
		Since $E(R/\p)$ is an $R_\p$-module by~\cite[Theorem 3.3.8(1)]{EJ} (whose proof works for any commutative ring), $E(R/\p)$ is an injective $R_{\p}$-module, and $E(R/\p)$ contains a copy of $\kappa(\p)$ as an essential submodule. Thus, $E(R/\p) = E(\kappa(\p))$ is an injective envelope of $\kappa(\p)$ and we have 
		$$\Hom{R}{\kappa(\p)}{E(R/\p)} \cong \Hom{R_\p}{\kappa(\p)}{E(\kappa(\p))} \cong \kappa(\p).$$
		Because $\Tor{i}{R}{M}{\kappa(\p)}$ is an $R_\p$-module, we can rewrite (\ref{eq:residue-duality}) as
		$$
		\Ext{i}{R}{M}{\kappa(\p)} \cong \Hom{R_\p}{\Tor{i}{R}{M}{\kappa(\p)}}{E(\kappa(\p))}.
		$$
		which concludes the proof, as $E(\kappa(\p))$ is a cogenerator for $\rmod{R_\p}$ (\cite[Proposition 18.15]{AF}).
\end{proof}

\begin{theorem}\label{main-1dim}
	The notion of a 1-tilting module is an ad-property.
\end{theorem}
\begin{proof}
		Let $\varphi: R \rightarrow S$ be a faithfully flat ring homomorphism and let $\tilde{T}$ be an $R$-module such that $T'=\tilde{T}\otimes_RS$ is a 1-tilting $S$-module. By Proposition~\ref{prop:partial-1-tilt} we know that $\tilde{T}$ is a partial 1-tilting module and we will prove that it is actually a 1-tilting module. Let $\lambda_{\tilde{T}}: R \rightarrow U$ be the tilting ring epimorphism associated to $\tilde{T}$ in the sense of the paragraph above. It is enough to show that $U=0$, because then $\mathcal{X}_{\tilde{T}}=0$, showing that $\tilde{T}$ is 1-tilting by \cite[Lemma 14.2]{GT}. 

		Next we use \cite[\S IV, Proposition 1.4]{L} to infer that $\lambda_{\tilde{T}}^*: \Spec{U} \xhookrightarrow{} \Spec{R}$ is an injective map with its image being equal to 
		$$\{\p \in \Spec{R} \mid \kappa(\p) \otimes_R U \neq 0\} = \{\p \in \Spec{R} \mid \kappa(\p) \otimes_R U \cong \kappa(\p) \}.$$ 
		In particular, by the identification of $\mathcal{X}_{\tilde T}$ with $\rmod{U}$,
		$$U \neq 0 \iff (\exists \p \in \Spec{R})(\kappa(\p) \in \mathcal{X}_{\tilde T}).$$
		If there is $\kappa(\p) \in \mathcal{X}_{\tilde T}$, then Lemma~\ref{lem:residue-duality} implies that $\Tor{i}{R}{\tilde T}{\kappa(\p)}=0$ for $i=0,1$. Since $\pd \tilde T \leq 1$, the latter vanishing actually holds for any $i \geq 0$, which is a contradiction with $\supph{\tilde T} = \Spec{R}$. Therefore, $U=0$, and thus we conclude that $\tilde{T}$ is a 1-tilting $R$-module.
\end{proof}


\begin{thebibliography}{EGPT}
		\bibitem{AF}
				{F. W. Anderson, K. R. Fuller}, \textit{Rings and categories of modules}, 2nd ed., Springer, Berlin 1992.

\bibitem{AH}
{L. Angeleri H\"ugel, D. Herbera}, \textit{Mittag-Leffler conditions on modules},
Indiana  Math. J. {\bf 57} (2008), 2459--2517.

\bibitem{APST}
{L. Angeleri H\" ugel, D. Posp\'\i\v sil, J. \v S\v tov\'\i\v cek, J. Trlifaj},
\textit{Tilting, cotilting, and spectra of commutative noetherian rings},
Trans.\ Amer.\ Math.\ Soc. {\bf 366} (2014), 3487--3517.

\bibitem{CE}
{H. Cartan, S. Eilenberg}, \textit{Homological Algebra}, Princeton
Univ. Press, Princeton 1956.

\bibitem{CTT}
		{R. Colpi, A. Tonolo, J. Trlifaj}. \textit{Perpendicular categories of infinite dimensional partial tilting modules and transfers of tilting torsion classes}, J. Pure Appl. Algebra {\bf 211} (2007), 223--234.

\bibitem{D} {V.\ Drinfeld},
{\sl Infinite--dimensional vector bundles in algebraic geometry: an
introduction}, in The Unity of Mathematics, Birkh\"auser, Boston
2006, pp.\ 263--304.

\bibitem{EE} 
{E.\ E.\ Enochs and S.\ Estrada}, {\sl Relative homological
algebra in the category of quasi-coherent sheaves.} Adv.\ in Math.
{\bf 194}(2005), 284--295.

\bibitem{EJ}
E.\ E.\ Enochs, O. M. G. Jenda, \textit{Relative Homological Algebra}, 2nd revised and extended ed., Vol. 1, W. de Gruyter, Berlin 2011.

\bibitem{EGT} 
S.\ Estrada, P.\ Guil Asensio, J.\ Trlifaj,
\textit{Descent of restricted flat Mittag-Leffler modules and generalized vector bundles},
Proc. Amer. Math. Soc. 142(2014), 2973--2981.

\bibitem{EGPT} 
S.\ Estrada, P.\ Guil Asensio, M.\ Prest, J.\ Trlifaj,
\textit{Model category structures arising from Drinfeld vector bundles},
Advances in Math. 231(2012), 1417--1438.

\bibitem{Gr}
{J. P. C. Greenlees}, \text{First steps in brave new commutative algebra}, Contemporary Mathematics 436 (2007), 239--276.

\bibitem{GT}
R. G\"{o}bel, J. Trlifaj, \textit{Approximations and Endomorphism Algebras of Modules}, 2nd revised and extended ed., W. de Gruyter, Berlin 2012.

\bibitem{H}
{M. Hochster}, \textit{Prime ideal structure in commutative rings}, Trans. Amer. Math. Soc. 142 (1969), 43--60.

\bibitem{Hrb}
{M.\ Hrbek}, \textit{One-tilting classes and modules over commutative rings}, J.\ Algebra 462 (2016), 1--22.

\bibitem{HS}
{M. Hrbek, J. \v{S}\'{t}ov\'{i}\v{c}ek}, \textit{Tilting classes over commutative rings}, arXiv preprint arXiv:1701.05534 (2017).


\bibitem{KP}
{J. Kock, W. Pitsch}, \textit{Hochster duality in derived categories and point-free reconstruction of schemes}, Trans. Amer. Math. Soc. 369 (2017), 223--261.

\bibitem{L}
{D. Lazard}, \text{Autour de la platitude}, Bull. Soc. Math. France \textbf{97} (1969), 81--128.

\bibitem{MS}
{F. Marks, J. \v{S}\'{t}ov\'{i}\v{c}ek}, \textit{Universal localisations via silting}, to appear in Proc.\ Roy.\ Soc.\ Edinburgh Sect.~A., arXiv preprint arXiv:1605.04222 (2016).

\bibitem{SP} 
{J.\ de Jong et al.}, 
{\sl The Stacks Project}, Version 118f702.

\bibitem{M}
{H. Matsumura},
\textit{Commutative Ring Theory}, CSAM \textbf{8}, Cambridge Univ.\ Press, Cambridge 1994.

\bibitem{N}
{A. Neeman, M. Bökstedt},
\textit{The chromatic tower for D(R)}, Topology \textbf{31}(3), Elsevier, 1992, 519--532.

\bibitem{PS}
{L. Positselski, J. \v{S}\'{t}ov\'{i}\v{c}ek}, \textit{Tilting-cotilting correspondence}, arXiv preprint arXiv:1710.02230 (2017).

\bibitem{RG}
{M. Raynaud, L.\ Gruson}, \emph{Crit\`eres de platitude et de
projectivit\'e}, Invent. Math. 13(1971), 1--89.

\bibitem{Sch}
{P.\ Schenzel}, \textit{Proregular sequences, local cohomology, and completion}, Math.\ Scand.\ \textbf{92} (2003), 161--180.
 
\bibitem{V}
R. Vakil, \textit{Math 216: Foundations of Algebraic Geometry}, available at
http://math.stanford.edu/~vakil/216blog/FOAGjun1113public.pdf.

\end{thebibliography}
\end{document}